\pdfoutput=1
\documentclass[12pt]{article}
\usepackage[margin=25mm]{geometry}
\usepackage{indentfirst}
\usepackage{amsfonts,amsmath,amssymb,amsbsy,amsthm}
\usepackage[usenames,dvipsnames,svgnames,table]{xcolor}
\usepackage[T1]{fontenc}
\usepackage{hyperref}
\hypersetup{ 
unicode=true,
colorlinks=true,
linkcolor={black},
citecolor={black},
urlcolor={blue!60!black},
pdftitle={Clustered Coloring of Graphs Excluding a Minor or Odd Minor},
pdfauthor={Chun-Hung Liu and David~R.~Wood}}
\usepackage{enumitem}
\setlist[itemize]{itemsep=0ex}
\setlist[enumerate]{itemsep=0ex}
\setlist[description]{itemsep=0ex}
\usepackage[longnamesfirst,numbers,sort&compress]{natbib}
\makeatletter
\def\NAT@spacechar{~}
\makeatother
\usepackage[noabbrev,capitalise]{cleveref}
\crefname{lem}{Lemma}{Lemmas}
\crefname{thm}{Theorem}{Theorems}
\crefname{prop}{Proposition}{Propositions}
\crefname{conj}{Conjecture}{Conjectures}
\crefname{claim}{Claim}{Claims}
\crefformat{equation}{(#2#1#3)}
\Crefformat{equation}{Equation #2(#1)#3}
\newtheorem{theorem}{Theorem}
\newtheorem{lemma}[theorem]{Lemma}

\newtheorem{corollary}[theorem]{Corollary}
\newtheorem{claim}{Claim}[theorem]
\newcommand{\ngs}[2][]{N_{#1}^{\geq s}(#2)}
\newcommand{\nls}[2][]{N_{#1}^{< s}(#2)}
\newcommand{\vdHW}{van~den~Heuvel and the second author~\cite{vdHW18}}
\newcommand{\VdHW}{Van~den~Heuvel and the second author~\cite{vdHW18}}
\newcommand{\VdHWcite}{Van~den~Heuvel and Wood~\cite{vdHW18}}
\newcommand{\VdHLS}{Van~der~Holst, Lov{\'a}sz and Schrijver~\cite{HLS}}
\makeatletter
\renewcommand\section{\@startsection {section}{1}{\z@}{-3ex \@plus -1ex \@minus -.2ex}{2ex \@plus.2ex}{\normalfont\large\bfseries}}
\renewcommand\subsection{\@startsection{subsection}{2}{\z@}{-2.5ex\@plus -1ex \@minus -.2ex}{1.5ex \@plus .2ex}{\normalfont\normalsize\bfseries}}
\renewcommand\subsubsection{\@startsection{subsubsection}{3}{\z@}{-2ex\@plus -1ex \@minus -.2ex}{1ex \@plus .2ex}{\normalfont\normalsize\bfseries}}
 \renewcommand\paragraph{\@startsection{paragraph}{4}{\z@}{1.5ex \@plus.5ex \@minus.2ex}{-1em}{\normalfont\normalsize\bfseries}}
\renewcommand\subparagraph{\@startsection{subparagraph}{5}{\parindent}  {1.5ex \@plus.5ex \@minus .2ex}  {-1em} {\normalfont\normalsize\bfseries}}
\makeatother
\def\X {{\mathcal X}}
\def\T {{\mathcal T}}
\def\Se {{\mathcal S}}
\def\L {{\mathcal L}}
\def\F {{\mathcal F}}
\def\V {{\mathcal V}}

\newcommand{\ceil}[1]{\lceil{#1}\rceil}

\renewcommand{\geq}{\geqslant}
\renewcommand{\leq}{\leqslant}
\newcommand\abs[1]{\lvert #1\rvert}
\begin{document}

\title{\bf\Large 
Clustered Coloring of Graphs\\ Excluding a Subgraph and a Minor\footnote{This material is based upon work supported by the National Science Foundation under Grant No.\ DMS-1664593, DMS-1929851 and DMS-1954054.}}

\author{Chun-Hung Liu\footnote{Department of Mathematics, Texas A\&M University, Texas, USA, \texttt{chliu@math.tamu.edu}. Partially supported by NSF under Grant No.\ DMS-1664593, DMS-1929851 and DMS-1954054.} \quad 
David R. Wood\footnote{School of Mathematics, Monash University, Melbourne, Australia, \texttt{david.wood@monash.edu}. Research supported by the Australian Research Council.}}

\maketitle

\begin{abstract}
A graph coloring has bounded clustering if each monochromatic component has bounded size. Equivalently, it is a partition of the vertices into induced subgraphs with bounded size components. This paper studies clustered colorings of graphs, where the number of colors depends on an excluded minor and/or an excluded subgraph. We prove the following results (for fixed integers $s,t$ and a fixed graph $H$). First we show that graphs with no $K_{s,t}$ subgraph and with no $H$-minor are $(s+2)$-colorable with bounded clustering. The number of colors here is best possible. This result implies that graphs with no $K_{s+1}$-minor are $(s+2)$-colorable with bounded clustering, which is within two colors of the clustered coloring version of Hadwiger's conjecture. For graphs of bounded treewidth (or equivalently, excluding a planar minor) and with no $K_{s,t}$ subgraph, we prove $(s+1)$-choosability with bounded clustering, which is best possible. We then consider excluding an odd minor. We prove that graphs with no $K_{s,t}$ subgraph and with no odd $H$-minor are $(2s+1)$-colorable with bounded clustering, generalizing a result of the first author and Oum who proved the case $s=1$. Moreover, at least $s-1$ color classes are stable sets. Finally, we consider the clustered coloring version of a conjecture of Gerards and Seymour and prove that graphs with no odd $K_{s+1}$-minor are $(8s-4)$-colorable with bounded clustering, which improves on previous such bounds. 
\end{abstract}

%
\section{Introduction}
\label{Intro}

Hadwiger's conjecture~\citep{Hadwiger43} asserts that every graph with no $K_{s+1}$-minor has a proper $s$-coloring. For $s\leq 2$ the conjecture is easy. \citet{Hadwiger43} and \citet{Dirac52} independently proved the $s=3$ case. \citet{Wagner37} proved that Hadwiger's conjecture with $s=4$ is equivalent to the Four Color Theorem~\citep{RSST97}. And \citet*{RST-Comb93} proved Hadwiger's conjecture for $s=5$. The conjecture remains open for $s\geq6$. Hadwiger's conjecture is widely considered to be one of the most important open problems in graph theory. 
The best known upper bound on the chromatic number of $K_{s+1}$-minor-free graphs is $O(s\log\log s)$ due to \citet{PD21}, improving a recent breakthrough of \citet{NPS20} who improved a long-standing bound independently due to \citet{Kostochka82,Kostochka84} and \citet{Thomason84,Thomason01}. 
Indeed, it is open whether  every graph with no $K_{s+1}$-minor is  $O(s)$-colorable. See the recent survey by \citet{SeymourHC} for more on Hadwiger's conjecture.


One  way to approach Hadwiger's conjecture is to allow improper colorings. Say that a \emph{coloring} of a graph $G$ is simply a function that assigns one color to each vertex of $G$. A \emph{monochromatic component} with respect to a coloring of $G$ is a connected component of the subgraph of $G$ induced by all the vertices assigned a single color. A coloring has \emph{clustering} $\eta$ if every monochromatic component has at most $\eta$ vertices. The \emph{clustered chromatic number} of a graph class $\mathcal{G}$ is the minimum integer $k$ for which there exists an integer $\eta$ such that every graph in $\mathcal{G}$ is $k$-colorable with clustering $\eta$. There have been several recent papers on this topic \citep{Kawa08,KM07,LMST08,HST03,ADOV03,CE19,vdHW18,KO19,EJ14,EO16,DN17,HW19,LO17,MRW17,NSSW19}; see \citep{WoodSurvey} for a survey. 

\citet{KM07} first proved a $O(s)$ upper bound on the clustered chromatic number of $K_{s+1}$-minor-free graphs. The number of colors has since been steadily improved, as shown in \cref{SuccessiveImprovements}, where $\eta(s)$ is some large unspecified function.

\setlength{\tabcolsep}{0pt}
\begin{table}[!h]
\caption{\label{SuccessiveImprovements} Clustered coloring of $K_{s+1}$-minor-free graphs }
\begin{center}
\begin{tabular}{lccc}
\hline
\multicolumn{2}{r}{ number of colors} & \quad clustering & \quad choosability \\
  \hline
\citet{KM07} &  $\ceil{\frac{31}{2}(s+1)}$ &  $\eta(s)$ & yes  \\
\citet{Wood10}\,\footnotemark[1] & $\ceil{\frac{7s+4}{2}}$ & $\eta(s)$ & yes  \\
\citet{EKKOS15} &  $4s$ & $\eta(s)$ & \\
\citet{LO17} & $3s$ & $\eta(s)$ & \\
\citet{Norin15}\,\footnotemark[2] & $2s$ & $\eta(s)$ & \\
\VdHWcite\ & $2s$ & $\lceil{\frac{s-2}{2}}\rceil$ & \\
\citet{DN17} & $2s$ & $\eta(s)$ &  \\
  \hline
\end{tabular}
\end{center}
\end{table}

\footnotetext[1]{This result depended on a result announced by Norine and Thomas~\citep{NT08,Thomas09} which has not yet been written.}
\footnotetext[2]{See~\citep{SeymourHC} for some of the details.}
\setcounter{footnote}{2}

It remains open whether graphs with no $K_{s+1}$ minor are $s$-colorable with bounded clustering\footnote{\citet{DN17} have announced that a forthcoming paper proves that graphs with no $K_{s+1}$ minor are $s$-colorable, in fact $s$-choosable, with bounded clustering.}. Note that $s$ colors would be best possible for any fixed clustering value. That is, for all $s\geq 2$ and $\eta$ there is a graph $G$ with no $K_{s+1}$ minor such that every $(s-1)$-coloring of $G$ has a monochromatic component with more than $\eta$ vertices\footnote{\citet{EKKOS15} 
proved the following stronger lower bound: for all $s\geq 2$ and $c$ there is a graph $G$ with no $K_{s+1}$ minor such that every $(s-1)$-coloring of $G$ has a monochromatic component with maximum degree greater than  $c$. Conversely, \citet{EKKOS15} proved that every graph with no $K_{s+1}$ minor is $s$-colorable such that each monochromatic component has maximum degree $O(s^2\log s)$. This degree bound was improved to $O(s)$ by \vdHW.}. In the following discussion we postpone giving standard definitions until \cref{Definitions}.

\subsection{Main Results}

The current best known bound on the clustered chromatic number of $K_{s+1}$-minor-free graphs is $2s$~\citep{Norin15,DN17,vdHW18}. We prove the following bound, which is within two colors of best possible. 

\begin{theorem}
\label{ClusteredMinor}
For every $s\in\mathbb{N}$, there exists $\eta\in\mathbb{N}$ such that every graph with no $K_{s+1}$-minor is $(s+2)$-colorable with clustering $\eta$.
\end{theorem}

We in fact prove the following stronger result where the number of colors only depends on an excluded $K_{s,t}$ subgraph. Indeed, the number of colors only depends on $s$. The dependence on $t$ and the excluded minor is hidden in the clustering function. 

\begin{theorem}
\label{minorbasic}
For all $s,t\in\mathbb{N}$ and for every graph $H$, there exists $\eta\in\mathbb{N}$ such that every graph with no  $H$ minor and with no $K_{s,t}$ subgraph is $(s+2)$-colorable with clustering $\eta$. 
\end{theorem}

\cref{minorbasic} will be proved in \cref{MainProofs}.

Since every graph with no $K_{s+1}$ minor has no $K_{s,s}$ subgraph, \cref{ClusteredMinor} is an immediate corollary of \cref{minorbasic}. This theorem and those in our companion papers \citep{LW1,LW3} are the first known results for clustered coloring where the number of colors depends on an excluded subgraph. While \cref{ClusteredMinor} is of substantial interest, we emphasise that our main results are for graph classes excluding a $K_{s,t}$ subgraph. One motivation for this line of research is that a graph contains no $K_{1,t}$ subgraph if and only if it has maximum degree less than $t$. So \cref{minorbasic} generalizes a result by the first author and Oum \cite{LO17} who proved the $s=1$ case which was originally conjectured by \citet{EJ14}. Also note that excluding a non-forest subgraph alone is not enough to guarantee bounded clustered chromatic number. In particular, for every graph $H$ that contains a cycle, and for all $k,\eta\in\mathbb{N}$, if $G$ is a graph with chromatic number greater than $k\eta$ and girth greater than $|V(H)|$ (which exists \citep{Erdos59}), then $G$ contains no $H$ subgraph and $G$ is not $k$-colorable with clustering $\eta$, for otherwise $G$ would be $k\eta$-colorable.

While \cref{ClusteredMinor} is within two colors of the conjectured answer, we now show that the number of colors in \cref{minorbasic} is best possible. The proof is a variation on the well known ``standard'' example; see \citep{WoodSurvey}. We claim that for all $s,\eta\in\mathbb{N}$ there is a graph $G_s$ with no $K_{s+4}$ minor and with no $K_{s,s+6}$ subgraph, such that every $(s+1)$-coloring of $G_s$ has a monochromatic component on at least $\eta$ vertices. We proceed by induction on $s$. In the base case, $s=1$, let $G_1$ be the $\eta\times\eta$ triangular grid graph. Then $G_1$ has no $K_5$ minor since it is planar, and $G_1$ has no $K_{1,7}$ subgraph since it has maximum degree 6. By the Hex Lemma~\citep{Gale79}, every 2-coloring of $G_1$ has a monochromatic path on $\eta$ vertices, as claimed. Now assume the claim for $G_{s-1}$. Let $G_s$ be obtained from $\eta$ disjoint copies of $G_{s-1}$ by adding a new vertex $v$ adjacent to all other vertices. 
Each component of $G_s-v$ is a copy of $G_{s-1}$. If $G_s$ contains a $K_{s+4}$ minor, then some component of $G_s-v$ contains a $K_{s+3}$ minor, which is a contradiction. Thus $G_s$ contains no $K_{s+4}$ minor. Similarly, if $G_s$ contains a $K_{s,s+6}$ subgraph, then $G_s-v$ contains a $K_{s-1,s+6}$ or $K_{s,s+5}$ subgraph, both of which contain $K_{s-1,(s-1)+6}$, which is a contradiction. Thus $G_s$ contains no $K_{s,s+6}$ subgraph. Now consider an $(s+1)$-coloring of $G_s$. Say $v$ is blue. If every component of $G_s-v$ has a blue vertex, then the blue subgraph contains a star on $\eta+1$ vertices, and we are done. Otherwise, some component $X$ of $G_s-v$ has no blue vertex, and thus has only $s$ colors. By induction, $X$ and hence $G_s$ contains a monochromatic component with at least $\eta$ vertices, as desired.

\subsection{Colin de Verdi\'{e}re Parameter}

The  Colin de Verdi\`ere parameter $\mu(G)$ is an important minor-closed graph invariant introduced by \citet{CdV90,CdV93}; see \citep{HLS,Schrijver97} for surveys. It is known that $\mu(G)\leq 1$ if and only if $G$ is a union of disjoint paths, $\mu(G)\leq 2$ if and only if $G$ is outerplanar, $\mu(G)\leq 3$ if and only if $G$ is planar, and  $\mu(G)\leq 4$ if and only if $G$ is linklessly embeddable. A famous conjecture of \citet{CdV90} asserts that every graph $G$ with $\mu(G)\leq s$ is properly $(s+1)$-colorable. 
This implies the Four Color Theorem, and is implied by Hadwiger's conjecture. It is open whether every graph $G$ with $\mu(G)\leq s$ is $(s+1)$-colorable with bounded clustering. 
Every graph $G$ with $\mu(G)\leq s$ contains no $K_{s+2}$ minor. 
So \cref{ClusteredMinor} implies that such graphs are $(s+3)$-colorable with bounded clustering, but a better bound can be obtained by \cref{minorbasic}. 
\VdHLS\ proved that $\mu(K_{s,t}) = s+1$ for $t\geq\max\{s,3\}$. Thus if $\mu(G)\leq s$ then $G$ contains no $K_{s,t}$ subgraph (since $\mu$ is monotone under taking subgraphs). \cref{minorbasic} then implies:

\begin{corollary}
For every $s\in\mathbb{N}$ there exists $\eta\in\mathbb{N}$, such that every graph $G$ with $\mu(G)\leq s$ is $(s+2)$-colorable with clustering $\eta$.
\end{corollary}

This example highlights the utility of excluding a subgraph within a minor-closed class.

\subsection{Bounded Treewidth}

When the excluded minor $H$ is planar (or equivalently, when the graph has bounded treewidth), \cref{minorbasic} is improved as follows. 

\begin{theorem} \label{twbasic}
For all $s,t,w\in\mathbb{N}$ there exists $\eta\in\mathbb{N}$ such that every graph with treewidth at most $w$ and with no $K_{s,t}$ subgraph is $(s+1)$-colorable with clustering $\eta$. 
\end{theorem}

The number of colors in \cref{twbasic} is best possible: for all $s,c\in\mathbb{N}$ there is a graph $G$ with treewidth $s$, with no $K_{s,s+2}$ subgraph, and such that every $s$-coloring of $G$ has a monochromatic component on at least $c$ vertices. The construction is analogous to the construction above except that in the base case ($s=1$) we use a long path instead of the triangular grid. This is called a ``standard'' example in \citep{WoodSurvey}.

We actually prove the following list-coloring result, which immediately implies \cref{twbasic}.

\begin{theorem}
\label{twmain}
For all $s,t,w\in\mathbb{N}$, there exists $\eta\in\mathbb{N}$ such that every graph with treewidth at most $w$ and with no $K_{s,t}$ subgraph is $(s+1)$-choosable with clustering $\eta$.
\end{theorem}

\cref{twmain} will be proved in \cref{TanglesTreewidth}.

The case $s=1$ of \cref{twmain} is an unpublished result of the first author (see \cite[Theorem 6.4]{SeymourHC}), which generalizes a result of \citet*{ADOV03} who proved \cref{twbasic} in the case $s=1$ (with much better bounds on $\eta$). 

\cref{twmain} immediately implies results for graphs with bounded treewidth and with no $K_{s,t}$-minor, although we emphasise that \cref{twmain} holds in the stronger setting of an excluded $K_{s,t}$ subgraph. 
In particular, \cref{twmain}  implies that graphs with bounded treewidth and with no $K_{s,t}$ minor are $(s+1)$-choosable with bounded clustering. Since $K_{s+1}$ is a minor of $K_{s,s}$, this in turn implies that graphs with bounded treewidth and with no $K_{s+1}$-minor are $(s+1)$-choosable with bounded clustering. \citet{DN17} proved this result with one fewer color.  That is,  graphs with bounded treewidth and with no $K_{s+1}$-minor are $s$-choosable with bounded clustering. In our companion paper \cite{LW3}, we strengthen this result by showing that 
graphs with bounded treewidth and with no $K_{s+1}$-topological-minor are $s$-choosable with bounded clustering. This says that a clustered version of Haj\'os' conjecture holds for bounded treewidth graphs, and even holds for choosability.
Results in this paper are critical components for \cite{LW3}.

\subsection{Excluded Odd Minors}

Gerards and Seymour (see \citep[\S6.5]{JT95}) conjectured that every graph with no odd $K_{s+1}$ minor is properly $s$-colorable, which implies Hadwiger's conjecture. 
The best known upper bound on the chromatic number of graphs with no odd $K_{s+1}$ minor is $O(s(\log\log s)^6)$, due to \citet{Postle20c}, improving earlier results of \citet*{GGRSV09} and \citet{NorinSong19b}. 
It is open whether such graphs are properly $O(s)$ colorable. The first $O(s)$ bound on the clustered chromatic number was established by \citet{Kawa08}, who proved that every graph with no odd $K_s$ minor is $496s$-colorable with bounded clustering. The number of colors was improved to $10s-13$ by \citet{KO19}. We make the following modest improvement. 

\begin{theorem}
\label{OddMinor}
For all $s\in\mathbb{N}$ there exists $\eta\in\mathbb{N}$ such that every graph with no odd $K_{s+1}$ minor is $(8s-4)$-colorable with clustering $\eta$. 
\end{theorem}

\cref{OddMinor} will be proved in \cref{GerardsSeymour}.

More interestingly, we prove the following analogue of \cref{minorbasic} for excluded odd minors and excluded subgraphs. 

\begin{theorem}
\label{oddminorbasic}
For all $s,t\in\mathbb{N}$ and for every graph $H$ there exists $\eta\in\mathbb{N}$ such that every graph with no odd $H$-minor and with no $K_{s,t}$ subgraph is $(2s+1)$-colorable with clustering $\eta$. Moreover, at least $s-1$ color classes are stable sets.
\end{theorem}

\cref{oddminorbasic} will be proved in \cref{MainProofs}.

The case of $s=1$ in \cref{oddminorbasic} was proved by the first author and Oum \cite{LO17}. Here, three colors is best possible. 

Note that no clustered choosability result is possible for graphs excluding an odd minor, since the complete bipartite graph $K_{n,n}$ contains no odd $K_3$ minor, but it follows from the work of \citet{Kang13} that for all $k,\eta\in\mathbb{N}$ there exists $n\in\mathbb{N}$ such that $K_{n,n}$ is not $k$-choosable with clustering $\eta$.

\subsection{\boldmath $K_{s,t}$-Minor-Free Graphs}

Consider graphs with no $K_{s,t}$ minor for $s\leq t$. \VdHW\ observed that results of \citet{EKKOS15} and Ossona de Mendez, Oum and the second author \cite{OOW19} imply that such graphs are $3s$-colorable with bounded clustering, which was improved to $2s+2$ by \citet{DN17}.  
\cref{minorbasic} immediately implies the following further improvement:

\begin{corollary}
\label{KstMinorFree}
For all $s,t\in\mathbb{N}$, there exists $\eta\in\mathbb{N}$ such that every graph with no $K_{s,t}$-minor is $(s+2)$-colorable with clustering $\eta$.
\end{corollary}

The best known lower bound on the clustered chromatic number of $K_{s,t}$-minor-free graphs is $s+1$, due to \vdHW. It is open whether every $K_{s,t}$-minor-free graph is $(s+1)$-colorable with bounded clustering. \VdHW\ proved this in the $s=2$ case. 

\subsection{\boldmath $H$-Minor-Free Graphs}

Now consider the clustered chromatic number of $H$-minor-free graphs, for an arbitrary graph $H$. A \emph{vertex cover} of  a graph $H$ is a set $S\subseteq V(H)$ such that $H-S$ has no edges. Suppose that $H$ has a vertex cover of size $s$. Then $H$ is a minor of $K_{s,|V(H)|-1}$ (obtained by contracting a matching of size $s-1$ in $K_{s,|V(H)|-1}$). So every graph containing no $H$-minor contains no $K_{s,|V(H)|-1}$-minor. \cref{KstMinorFree} thus implies:

\begin{corollary}
\label{VertexCover}
For all $s \in {\mathbb N}$ and for every graph $H$ that has a vertex-cover of size at most $s$, 
there exists $\eta \in \mathbb{N}$ such that every graph with no $H$-minor is $(s+2)$-colorable with clustering $\eta$.
\end{corollary}

\newcommand{\ctd}{\overline{\text{td}}}

We now relate this result to a conjecture of Norin, Scott, Seymour and the second author~\cite{NSSW19} about the clustered chromatic number of   $H$-minor-free graphs. Let $T$ be a rooted tree. The \emph{depth} of $T$ is the maximum number of vertices on a root--to--leaf path in $T$. The \emph{closure} of $T$ is obtained from $T$ by adding an edge between every ancestor and descendent in $T$. The \emph{connected tree-depth} of a graph $H$, denoted by $\ctd(H)$, is the minimum depth of a rooted tree $T$ such that $H$ is a subgraph of the closure of $T$. \citet{NSSW19} observed that for every graph $H$ and $\eta\in\mathbb{N}$ there is an $H$-minor-free graph that is not $(\ctd(H)-2)$-colorable with clustering $\eta$; thus the clustered chromatic number of $H$-minor-free graphs is at least $\ctd(H)-1$. On the other hand, \citet{NSSW19} conjectured that the class of  $H$-minor-free graphs has clustered chromatic number at most $2\,\ctd(H)-2$, which would be tight for certain graphs $H$. As evidence for this conjecture, \citet{NSSW19} proved that the clustered chromatic number  of $H$-minor-free graphs is at most $2^{\ctd(H)+1}-4$.

For $h \in {\mathbb N}$ with $h \geq 3$, a \emph{broom} of height $h$ is a rooted tree that can be obtained from a star rooted at a leaf by subdividing the edge incident with the root $h-3$ times. 
The closure of a broom of height $h$ has a vertex-cover of size at most $h-1$. 
Thus \cref{VertexCover} can be restated as follows:

\begin{corollary}
\label{HMinorFree}
For every integer $h\geq 3$, if $H$ is a subgraph of the closure of the broom of height $h$, then there exists $\eta \in \mathbb{N}$ such that every graph with no $H$-minor is $(h+1)$-colorable with clustering $\eta$.
\end{corollary}

Note that the depth of the broom of height $h$ is $h$. 
So \cref{HMinorFree} answers the aforementioned conjecture of \citet{NSSW19} in a stronger sense (since $h+1 \leq 2h-2$) when the underlying tree $T$ is a broom.

\subsection{Excluded Subdivisions}

Our companion paper \citep{LW3} studies clustered colourings of graphs excluding various graph subdivisions. The methods build heavily on those introduced in this paper and reference \citep{LW3} uses some results in this paper as a black box. For example, we prove in \cite{LW3} that graphs with bounded treewidth and with no almost $(\leq 1)$-subdivision of $K_{s+1}$ are $s$-choosable with bounded clustering. 
Here a graph is an {\it almost $(\leq 1)$-subdivision} of a graph $H$ if it can be obtained from $H$ by subdividing edges, where at most one edge is subdivided more than once.
This result is a clustered choosability version of Haj\'os conjecture for graphs of bounded treewidth. 
Allowing one more colour, we prove an analogous result for graphs excluding a fixed minor; that is, we prove that graphs excluding a fixed graph as a minor and with no almost $(\leq 1)$-subdivision of $K_{s+1}$ are $(s+1)$-colorable with bounded clustering. 

\subsection{Standard Definitions}
\label{Definitions}

Let $\mathbb{N}:=\{1,2,\dots\}$ and $\mathbb{N}_0:=\{0,1,2,\dots\}$. For $m,n\in\mathbb{N}$, let $[m,n]:=\{m,m+1,\dots,n\}$ and $[n]:=[1,n]$. 

Let $G$ be a graph (allowing loops and parallel edges). For $v\in V(G)$, let $N_G(v):=\{w\in V(G): vw\in E(G)\}$ be the neighborhood of $v$, and let  $N_G[v] :=N_G(v) \cup\{v\}$. For $X\subseteq V(G)$, let $N_G(X):= \bigcup_{v\in X} (N_G(v) - X)$ and $N_G[X]:= N_G(X)\cup X$. Denote the subgraph of $G$ induced by $X$ by $G[X]$.

For our purposes, a \emph{color} is an element of $\mathbb{Z}$. A \emph{list-assignment} of a graph $G$ is a function $L$ with domain containing $V(G)$, such that $L(v)$ is a non-empty set of colors for each vertex $v\in V(G)$. For a list-assignment $L$ of $V(G)$, an \emph{$L$-coloring} of $G$ is a function $c$ with domain $V(G)$ such that $c(v) \in L(v)$ for every $v \in V(G)$. 
An $L$-coloring has \emph{clustering} $\eta$ if every monochromatic component has at most $\eta$ vertices. A list-assignment $L$ of a graph $G$ is an \emph{$\ell$-list-assignment} if $|L(v)|\geq\ell$ for every vertex $v\in V(G)$. 
A graph is \emph{$\ell$-choosable with clustering $\eta$} if $G$ is $L$-colorable with clustering $\eta$ for every $\ell$-list-assignment $L$ of $G$. 

A graph $H$ is a \emph{minor} of a graph $G$ if a graph isomorphic to $H$ can be obtained from a subgraph of $G$ by contracting edges. Here we allow $H$ to have loops and parallel edges. 
The following is an alternative definition of graph minor. Let $H$ be a graph. An {\it $H$-minor} of a graph $G$ is a map $\alpha$ with domain $V(H) \cup E(H)$ such that:
\begin{itemize}
	\item For every $h \in V(H)$, $\alpha(h)$ is a nonempty connected subgraph of~$G$ (called a \emph{branch set}), 
	\item If $h_1$ and $h_2$ are different vertices of~$H$, then $\alpha(h_1)$ and $\alpha(h_2)$ are disjoint.
	\item For each edge $e=h_1h_2$ of~$H$, $\alpha(e)$ is an edge of~$G$ with one end in $\alpha(h_1)$ and one end in $\alpha(h_2)$; furthermore, if $h_1=h_2$, then $\alpha(e) \in E(G)-E(\alpha(h_1))$. 
        \item If $e_1, e_2$ are distinct edges of~$H$, then $\alpha(e_1) \neq \alpha(e_2)$.
\end{itemize}
Then $\alpha$ is an {\it odd $H$-minor} if there exists a 2-coloring $c$ of $\bigcup_{h \in V(H)}\alpha(h)$ such that $c|_{\alpha(h)}$ is a proper 2-coloring of $\alpha(h)$, and for every edge $e$ of $H$, the ends of $\alpha(e)$ receive the same color in $c$. See \citep{Kawa08,KO19,GGRSV09,NorinSong19b,Postle20c} for work on odd minors. 

A \emph{tree decomposition} of a graph $G$ is a pair $(T,\mathcal{X}=(X_x:x\in V(T)))$, where $T$ is a tree and for each node $x\in V(T)$,  $X_x$ is a subset of $V(G)$ called a \emph{bag}, such that for each vertex $v\in V(G)$, the set $\{x\in V(T):v\in X_x\}$ induces a non-empty (connected) subtree of $T$, and for each edge $vw\in E(G)$ there is a node $x\in V(T)$ such that $v,w\in X_x$. A \emph{path decomposition} is a tree decomposition whose underlying tree is a path.
The \emph{width} of a tree decomposition $(T,\mathcal{X})$ is $\max\{|X_x|-1: x\in V(T)\}$. The \emph{treewidth} of a graph $G$ is the minimum width of a tree decomposition of $G$. For each integer $k$, the graphs with treewidth at most $k$ form a minor-closed class. \citet{RS-V} proved that a minor-closed class of graphs has bounded treewidth if and only if some planar graph is not in the class. treewidth is a key parameter in algorithmic and structural graph theory; see \citep{Reed97,Bodlaender-TCS98,HW17} for surveys. 

A \emph{separation} of a graph $G$ is an ordered pair $(A,B)$ of edge-disjoint subgraphs of $G$ with $A \cup B=G$. 
The \emph{order} of $(A,B)$ is $\lvert V(A \cap B) \rvert$. 

A {\it tangle} $\T$ in a graph $G$ of order $\theta \in {\mathbb Z}$ is a set of separations of $G$ of order less than $\theta$ such that the following hold:
	\begin{itemize}
		\item[(T1)] For every separation $(A,B)$ of $G$ of order less than $\theta$, either $(A,B) \in \T$ or $(B,A) \in \T$.
		\item[(T2)] If $(A_i,B_i) \in \T$ for $i \in [3]$, then $A_1 \cup A_2 \cup A_3 \neq G$.
		\item[(T3)] If $(A,B) \in \T$, then $V(A) \neq V(G)$.
	\end{itemize}

A \emph{surface} is a nonnull compact connected $2$-manifold without boundary. Every surface is homeomorphic to the sphere with $k$ handles (which has Euler genus $2k$) or the sphere with $k$ cross-caps (which has Euler genus $k$). The \emph{Euler genus} of a graph $G$ is the minimum Euler genus of a surface in which $G$ embeds; see \citep{MoharThom} for more on graph embeddings. 

\section{List Coloring Setup}
\label{Setup}

We prove \cref{minorbasic,oddminorbasic,twmain} using the same technique. A key is to actually prove stronger results that allow for a bounded-size set $Y$ of precolored vertices. We then require that not only every monochromatic component has bounded size, but also that the union of all the monochromatic components intersecting $Y$ has size at most $g(\lvert Y \rvert)$, for some function $g$. Assume that $G$ is a minimum counterexample, and subject to this, the size of $Y$ is as large as possible. We distinguish two cases depending on the size of $Y$. 

First consider the case that $Y$ is large. Let $\theta$ be a large number. If there exists a separation $(A,B)$ of $G$ of order less than $\theta$ such that both $V(A) \cap Y$ and $V(B) \cap Y$ contains at least $3\theta$ vertices, then we can precolor the vertices in $V(A \cap B)$ so that the number of precolored vertices in $A$ (and $B$, respectively) is smaller than $|Y|$, apply induction to each of $A$ and $B$ with the new precolored set to obtain a coloring of $A$ and a coloring of $B$, and then combine the colorings to obtain a coloring of $G$.  So we may assume that such a separation does not exist. This defines a tangle of order $\theta$. But such a tangle does not exist when the graph has bounded treewidth, which finishes this case for graphs of bounded treewidth (\cref{twmain}). 

For graphs excluding a minor (which might have unbounded treewidth), we apply Robertson and Seymour's Graph Minor Structure Theorem \citep{RS-XVI}, which describes the structure of graphs excluding a minor relative to a tangle of large order. For graphs excluding an odd minor, the extra ingredient is the structure theorem of \citet*{GGRSV09}. 
We then apply a result of Dujmovi\'{c}, Morin, and the second author~\cite{DMW17}, from which we (roughly) conclude that ignoring the apex vertices, our graph has bounded layered treewidth (defined below). We then apply a result in our companion paper \citep{LW1} that shows $(s+2)$-colorability with bounded clustering for graphs of bounded layered treewidth with no $K_{s,t}$-subgraph. From this we conclude the result. 

It remains to deal with the case that $Y$ is small. 

For the time being, assume that $s=1$; that is, $G$ has bounded maximum degree. 
Let $c_1,c_2,\dots,c_r$ be the colors appearing in $Y$, where $r \leq \lvert Y \rvert$.
First precolor $N(Y)$ so that no vertex uses $c_1$, then precolor $N(N(Y))$ so that no vertex uses $c_2$, and so on, until the $r$-th neighborhood of $Y$ is precolored so that no vertex uses $c_r$.
Let $Y'$ be the set of vertices at distance at most $r$ from $Y$. 
Since $Y$ is small and $G$ has bounded maximum degree, $|Y'|$ is bounded.
Apply induction to obtain a desired coloring of $G$ with $Y'$ precolored. 
Note that in this coloring of $G$, every monochromatic component intersecting $Y$ is contained in $Y'$ so the size is bounded by $g(\lvert Y \rvert)$.

However, this approach for enlarging $Y$ does not work directly when $s \geq 2$, 
since the precolored set might grow too fast when the maximum degree is unbounded. 
We employ the following alternative strategy. Instead of precoloring every vertex that is adjacent to the currently precolored set, only precolor those vertices that are adjacent to at least $s$ currently precolored vertices so  that they forbid one color in $Y$, and for each vertex $v$ that is adjacent to at least 1 but at most $s-1$ currently precolored vertices, we ensure (using a list coloring argument) that in the future $v$ is assigned a color that appears on no precolored neighbor of $v$. This allows us to enlarge $Y$ to obtain a larger precolored set $Y'$, such that in every coloring, every monochromatic component that intersects $Y$ is contained in $Y'$, so it has size less than $g(\lvert Y \rvert)$. \cref{BoundedGrowth} below, which is proved in our companion paper~\citep{LW1}, ensures that the size of the precolored set does not increase too much, which is then used to ensure that the final precolored set $Y'$ has bounded size.

The following definitions formalise these ideas. 
For a graph $G$, a set $X \subseteq V(G)$ and $s\in\mathbb{N}$, define
\begin{align*}
\ngs[G]{X} & := \{v \in V(G)- X:  \lvert N_G(v) \cap X \rvert \geq s\} \text{ and} \\
\nls[G]{X} & := \{v \in V(G)- X:  1 \leq \lvert N_G(v) \cap X \rvert < s\}.
\end{align*}
When the graph $G$ is clear from the context we write $\ngs{X}$ instead of $\ngs[G]{X}$, and similarly for 
$\nls{X}$.

\begin{lemma}[\cite{LW1}] 
\label{BoundedGrowth}
For all $s,t\in\mathbb{N}$, there exists a function $f_{s,t}:\mathbb{N}_0\rightarrow\mathbb{N}_0$ such that for every graph $G$ with no $K_{s,t}$ subgraph, if $X\subseteq V(G)$ then $| N^{\geq s}(X) |  \leq f_{s,t}(\lvert X \rvert)$.
\end{lemma}

When $G$ excludes a fixed minor, the function $f_{s,t}$ in \cref{BoundedGrowth} can be made linear; see \citep{LW1}. This improves the clustering function in all our results, but for ease of presentation we choose not to evaluate explicit clustering functions in this paper. 

All our results rely on the following list coloring setup. For $s,r\in\mathbb{N}$ and $Y_1\subseteq V(G)$, a list-assignment $L$ of a graph $G$ is an {\it $(s,r,Y_1)$-list-assignment} if the following hold:
	\begin{itemize} 
		\item[(L1)] $| L(v) | \in [s+r]$ for every $v \in V(G)$.
		\item[(L2)] $Y_1 = \{v \in V(G): \lvert L(v) \rvert=1\}$.
		\item[(L3)] For every $y \in \nls{Y_1}$, 
			\begin{equation*}\lvert L(y) \rvert = s+r-\lvert N_G(y) \cap Y_1 \rvert, \end{equation*} 
			and $L(y) \cap L(u)=\emptyset$ for every $u \in N_G(y) \cap Y_1$. (Note that 	$\lvert L(y) \rvert \geq r+1$.)
		\item[(L4)] For every $v \in V(G)-N_G[Y_1]$, we have $\lvert L(v) \rvert =s+r$.
		\item[(L5)] For every $v \in V(G)-Y_1$, we have $\lvert L(v) \rvert \geq r+1$.
	\end{itemize}
We use $r=1$ for the bounded treewidth case (\cref{twmain}), $r=2$ for excluded minors (\cref{minorbasic}), and $r=s+1$ for excluded odd minors (\cref{oddminorbasic}). Define an {\it $(s,Y_1)$-list-assignment} to be an $(s,1,Y_1)$-assignment.

Let $G$ be a graph, let $Y_1\subseteq V(G)$, let $s,r \in {\mathbb N}$, and let $L$ be an $(s,r,Y_1)$-list-assignment.
For all $W\subseteq V(G)$ and for every set $F$ of colors with $\lvert F \rvert \leq r$ (not necessarily a subset of $\bigcup_{v \in V(G)}L(v)$), 
a {\it $(W,F)$-progress} of $L$ is a list-assignment $L'$ of $G$ defined as follows:
\begin{itemize}
	\item Let $Y_1':=Y_1 \cup W$. 
	\item For every $y \in Y_1$, let $L'(y):=L(y)$.
	\item For every $y \in Y_1'-Y_1$, let $L'(y)$ be a 1-element subset of $L(y)-F$	 (which exists by (L5)). 
	\item For each $v \in \nls{Y_1'}$, let $L'(v)$ be a subset of $$L(v)-\bigcup_{w \in N_G(v) \cap (W-Y_1)}L'(w)$$ 
	of size $\lvert L(v) \rvert - \lvert N_G(v) \cap (W-Y_1) \rvert$ such that $\lvert L'(v) \cap F \rvert$ is as large as possible. 
	\item For every $v \in V(G)-(Y_1' \cup \nls{Y_1'})$, let $L'(v):=L(v)$.
\end{itemize}
Intuitively speaking, the $(W,F)$-progress is a list assignment where each uncolored vertex in $W$ is assigned a color in its list but not in $F$.

\begin{lemma}
\label{progress basic}
Let $G$ be a graph, $s,r \in {\mathbb N}$, $Y_1 \subseteq V(G)$, and $L$ be an $(s,r,Y_1)$-list-assignment.
If $W\subseteq V(G)$ and $F$ is a set of colors with $\lvert F \rvert \leq r$, then every $(W,F)$-progress $L'$ of $L$ satisfies the following properties:
	\begin{enumerate}
		\item $L'$ is an $(s,r,Y_1 \cup W)$-list-assignment of $G$.
		\item $L'(v) \subseteq L(v)$ for every $v \in V(G)$.
		\item $\{v \in Y_1 \cup W: L'(v) \cap F \neq \emptyset\} = \{v \in Y_1: L(v) \cap F \neq \emptyset\}$.
		\item If $\ngs{Y_1} \subseteq W$, then for every $y \in Y_1 \cup W$ and color $x \in F\cap L'(y)$, 
		we have $\{v \in N_G(y) -(Y_1 \cup W): x \in L'(v)\}=\emptyset$.
		\item For every $v \in V(G)-(Y_1 \cup W)$, we have $L'(v) \cap F = L(v) \cap F$.
	\end{enumerate} 
\end{lemma}

\begin{proof}
Let $L'$ be a $(W,F)$-progress of $L$.
By construction, Statements 2 and 3 hold. 
Let $Y_1':=Y_1 \cup W$.

Now we prove Statement 4.
Suppose to the contrary that there exist $y \in Y_1'$, $b \in N_G(y) -Y_1'$ and $f \in F\cap L'(b) \cap L'(y)$.
Since $f \not \in L'(q)$ for every $q \in Y_1'-Y_1$, we have $y \in Y_1$.
Since $\ngs{Y_1} \subseteq W$ and $b \not \in W$, we have $\lvert N_G(b) \cap Y_1 \rvert \in [s-1]$.
That is, $b \in \nls{Y_1}$.
However, $f \in L'(b) \cap L'(y) \subseteq L(b) \cap L(y)$, contradicting that $L$ satisfies (L3).
Hence Statement 4 holds.

Now we prove Statement 5.
Suppose there exists $a \in V(G)-Y_1'$ such that $L'(a) \cap F \neq L(a) \cap F$.
If $a \not\in \nls{Y_1'}$, then $L'(a)=L(a)$, a contradiction.
Thus $a \in \nls{Y_1'}$.
Since $Y_1' \supseteq Y_1$, either $a \in \nls{Y_1}$ or $a \in V(G)-N_G[Y_1]$.
In either case, $\lvert L(a) \rvert = s+r-\lvert N_G(a) \cap Y_1 \rvert$ by (L3) and (L4).
Hence
\begin{align*}
\lvert L'(a) \rvert = \lvert L(a) \rvert - \lvert N_G(a) \cap (W-Y_1) \rvert 
& = s+r-\lvert N_G(a) \cap Y_1 \rvert - \lvert N_G(a) \cap (W-Y_1) \rvert\\ 
& = s+r-\lvert N_G(a) \cap Y_1' \rvert \\
& \geq r+1 \\ 
& > \lvert F \rvert \\
& \geq \lvert L(a) \cap F \rvert. 
\end{align*}
Since $L'(a)$ is chosen so that $\lvert L'(a) \cap F \rvert$ is maximum among all subsets of \linebreak $L(v)- \bigcup_{w \in N_G(v) \cap (W-Y_1)}L'(w)$ of size $\lvert L(v) \rvert - \lvert N_G(v) \cap (W-Y_1) \rvert$, and $\{y \in W-Y_1: L'(y) \cap F \neq \emptyset\}=\emptyset$, we know $L'(a)$ must contain $L(a) \cap F$.
Since $L'(a)\subseteq L(a)$, we have $L'(a) \cap F \subseteq L(a) \cap F$.
Hence $L'(a) \cap F = L(a) \cap F$, a contradiction.
So Statement 5 holds.

To complete the proof it suffices to show that $L'$ is an $(s,r,Y_1')$-list-assignment of $G$.

We first show that $L'$ satisfies (L3).
Let $v \in \nls{Y_1'}$.
Since $Y_1 \subseteq Y_1'$, either $N_G(v) \cap Y_1 = \emptyset$ or $v \in \nls{Y_1}$.
If $N_G(v) \cap Y_1 = \emptyset$, then 
\begin{equation*}
\lvert L'(v) \rvert = \lvert L(v) \rvert - \lvert N_G(v) \cap (W-Y_1) \rvert = s+r - \lvert N_G(v) \cap Y_1' \rvert \geq r+1,
\end{equation*} 
and $L'(v)$ is a subset of $L(v)-\{L'(w): w \in N_G(v) \cap Y_1'\}$, so $L'(v) \cap L'(u)=\emptyset$ for every $u \in N_G(v) \cap Y_1'$.
So we may assume $v \in \nls{Y_1}$.
Since $L$ satisfies (L3), we have
$\lvert L(v) \rvert = s+r-\lvert N_G(v) \cap Y_1 \rvert \geq 2$ and $L(v) \cap L(u) = \emptyset$ for every $u \in N_G(v) \cap Y_1$.
Hence 
\begin{align*}
\lvert L'(v) \rvert & = \lvert L(v) \rvert - \lvert N_G(v) \cap (W-Y_1) \rvert\\ 
& = s+r-\lvert N_G(v) \cap Y_1 \rvert - \lvert N_G(v) \cap (W-Y_1) \rvert \\
& = s+r - \lvert N_G(v) \cap (W \cup Y_1) \rvert \\ 
&= s+r - \lvert N_G(v) \cap Y_1' \rvert.
\end{align*}
Furthermore, $L'(v) \subseteq L(v)-\{L'(w): w \in W-Y_1\}$ and $L'(u)=L(u)$ for every $u \in Y_1$, so $L'(v) \cap L'(u) = \emptyset$ for every $u \in N_G(v) \cap (W \cup Y_1) = N_G(v) \cap Y_1'$.
Hence $L'$ satisfies (L3).

Let $x$ be a vertex in $V(G)-N_G[Y_1']$.
Since $Y_1 \subseteq Y_1'$, we have $x \in V(G)-N_G[Y_1]$.
Since $L$ satisfies (L4), we have $\lvert L(x) \rvert = s+r$.
Since $Y_1' \cup N_G(Y_1') \supseteq Y_1' \cup \nls{Y_1'}$, we have $L'(x)=L(x)$ which has size $s+r$.
This shows that $L'$ satisfies (L4).

Let $z \in \ngs{Y_1'}$.
Then $\lvert L'(z) \rvert = \lvert L(z) \rvert$.
Since $z \not \in Y_1$ and $L$ satisfies (L2)--(L5), we have $\lvert L'(z) \rvert \geq r+1$. 
Since $L'$ satisfies (L3) and (L4), it implies that $L'$ satisfies (L5).
Since $L'$ satisfies (L3)--(L5), $L'$ satisfies (L1) and (L2).
Therefore $L'$ is an $(s,r,Y_1')$-list-assignment.
\end{proof}

Let $G$ be a graph, $\eta\in\mathbb{N}$, and $g$ a nondecreasing function.
Let $L$ be a list-assignment of $G$, and let $Y_1:=\{v \in V(G): \lvert L(v) \rvert=1\}$.
Then an $L$-coloring $c$ is {\it $(\eta,g)$-bounded} if:
	\begin{itemize}
		\item the union of the monochromatic components with respect to $c$ intersecting $Y_1$ contains at most $\lvert Y_1 \rvert^2 g(\lvert Y_1 \rvert)$ vertices, and 
		\item every monochromatic component with respect to $c$ contains at most $\eta^2 g(\eta)$ vertices.
	\end{itemize}

\begin{lemma} \label{enlarge Y}
For all $s,t,k\in\mathbb{N}$, there exist a number $\eta>k$ and a nondecreasing function $g$ with domain ${\mathbb N_0}$ and with $g(0) \geq \eta$ such that if $G$ is a graph with no $K_{s,t}$ subgraph, $r \in {\mathbb N}$, $Y_1$ is a subset of $V(G)$ with $\lvert Y_1 \rvert \leq \eta$, $F$ is a set of colors with $\lvert F \rvert \leq r-1$, and $L$ is an $(s,r,Y_1)$-list-assignment of $G$ such that $\{y \in Y_1: x \in L(y)\}$ is a stable set in $G$ for every $x \in F$, then one of the following holds:
	\begin{enumerate}
		\item There exists an $(\eta,g)$-bounded $L$-coloring of $G$ such that for every $x \in F$, the set of vertices colored $x$ is a stable set in $G$.
		\item $\lvert Y_1 \rvert >k$.
		\item For every color $\ell$, there exist a subset $Y_1'$ of $V(G)$ with $\eta \geq \lvert Y_1' \rvert > \lvert Y_1 \rvert$ and an $(s,r,Y_1')$-list-assignment $L'$ of $G$ with $L'(v) \subseteq L(v)$ for every $v \in V(G)$, such that:
			\begin{enumerate}
				\item there does not exist an $(\eta,g)$-bounded $L'$-coloring of $G$ such that for every $x \in F$, the set of vertices colored $x$ is a stable set in $G$, 
				\item for every $L'$-coloring of $G$, every monochromatic component intersecting $Y_1$ is contained in $G[Y_1']$,
				\item $\{v \in Y_1: F \cap L(v) \neq \emptyset\} = \{v \in Y_1': F \cap L'(v) \neq \emptyset\}$, 
				\item for every $x \in F \cup \{\ell\}$ and $y \in Y_1'$ with $x \in L'(y)$, we have $\{v \in N_G(y)-Y_1': x \in L'(v)\}=\emptyset$, and 
				\item for every $v \in V(G)-Y_1'$, we have $L'(v) \cap F = L(v) \cap F$.
			\end{enumerate}
		\item $Y_1 \neq \emptyset$, $N_G(Y_1)=\emptyset$, and there does not exist an $(\eta,g)$-bounded $L|_{G-Y_1}$-coloring of $G-Y_1$ such that for every $x \in F$, the set of vertices colored $x$ is a stable set in $G-Y_1$. 
	\end{enumerate}
\end{lemma}

\begin{proof}
Let $f$ be the function $f_{s,t}$ in \cref{BoundedGrowth}. 
Let $h_{-1}:\mathbb{N}_0\to \mathbb{N}_0$ be the identity function, and let $h_0:\mathbb{N}_0\to \mathbb{N}_0$ be the function defined by $h_0(x):=x+f(x)$ for $x\in\mathbb{N}_0$. 
For $i \geq 1$, let $h_i:\mathbb{N}_0\to \mathbb{N}_0$ be the function defined by $h_i(x):=h_0(h_{i-1}(x))$  for $x\in\mathbb{N}_0$.
Let $\eta:=\max\{h_{k}(k), k+1\}$.
Let $g:\mathbb{N}_0\to \mathbb{N}_0$ to be the function defined by $g(x):=h_{x}(x)+\eta$ for  $x\in\mathbb{N}_0$.

Suppose that Statements 1, 2 and 4 do not hold. 
Suppose to the contrary that Statement 3 does not hold for some color $\ell$. 

First suppose that $Y_1=\emptyset$.
Define $R$ to be a $(\{v\},F \cup \{\ell\})$-progress of $L$, where $v$ is a vertex of $G$.
If $G$ has an $(\eta,g)$-bounded $R$-coloring such that for every $x \in F$, the set of vertices colored $x$ is a stable set, then $G$ has an $(\eta,g)$-bounded $L$-coloring such that for every $x \in F$, the set of vertices colored $x$ is a stable set, so Statement 1 holds, a contradiction.
So $G$ has no $(\eta,g)$-bounded $R$-coloring such that for every $x \in F$, the set of vertices colored $x$ is a stable set, then Statement 3 holds if we take $Y_1'=\{v\}$ and $L'=R$ by \cref{progress basic}.

Hence $Y_1 \neq \emptyset$.
Suppose that $N_G(Y_1)=\emptyset$.
If $V(G)=Y_1$, then there exists an $(\eta,g)$-bounded $L$-coloring of $G$ such that for every $x \in F$, the set of vertices colored $x$ is a stable set, a contradiction.
So $V(G) \neq Y_1$.
Define $G'=G-Y_1$.
Since Statement 4 does not hold, there exists an $(\eta,g)$-bounded $L|_{G'}$-coloring of $G'$ such that for every $x \in F$, the set of vertices colored $x$ is a stable set.
We can further color the vertices in $Y_1$ by the unique element in their lists to obtain an $(\eta,g)$-bounded $L$-coloring of $G$ such that for every $x \in F$, the set of vertices colored $x$ is a stable set, a contradiction.

Hence $Y_1 \neq \emptyset$ and $N_G(Y_1) \neq \emptyset$.
Let $Y_2 = \nls{Y_1}$.

Suppose that $\ngs{Y_1}=\emptyset$.
So $Y_2=N_G(Y_1)$.
Since $N_G(Y_1) \neq \emptyset$, there exists a vertex $z \in Y_2$.
Define $Y_1':=Y_1 \cup \{z\}$ and define $L'$ to be a $(\{z\},F \cup \{\ell\})$-progress of $L$. 
Since Statement 2 does not hold, $\lvert Y_1' \rvert = \lvert Y_1 \rvert+1 \leq k+1 \leq \eta$.
Since $Y_2=N_G(Y_1)$ and by (L3) $L(v) \cap L(u)=\emptyset$ for every $v \in Y_2$ and $u \in N_G(v) \cap Y_1$, 
the union of the monochromatic components with respect to any $L'$-coloring of $G$ intersecting $Y_1$ is contained in $G[Y_1]$.
Since Statement 3 does not hold, by \cref{progress basic}, there exists an $(\eta,g)$-bounded $L'$-coloring $c'$ of $G$ such that for every $x \in F$, the set of vertices colored $x$ is a stable set.
In particular, the union of the monochromatic components with respect to $c'$ intersecting $Y_1$ is contained in $G[Y_1]$.
So $c'$ is an $(\eta,g)$-bounded $L$-coloring of $G$, and Statement 1 holds, a contradiction.

Hence $\ngs{Y_1}\neq\emptyset$.
Denote $Y_1$ by $\{y_1,y_2,\dots,y_{\lvert Y_1 \rvert}\}$.
For every $i \in [\lvert Y_1 \rvert]$, let $\ell_i$ be the unique element of $L(y_i)$; let $\ell_{\lvert Y_1 \rvert+1}=\ell$.
Define $L_0=L$ and $U_0=Y_1$.
For $i \geq 1$, let $L_i$ be a $(\ngs{U_{i-1}},\{\ell_i\} \cup F)$-progress and define $U_i := \{v \in V(G): \lvert L_i(v) \rvert=1\}$.
Define $L^*:=L_{\lvert Y_1 \rvert+1}$ and $Y_1^* := U_{\lvert Y_1 \rvert+1}$.

\begin{claim}\label{EnlargeClaim1}
$\lvert Y_1^* \rvert \leq h_{\lvert Y_1 \rvert}(\lvert Y_1 \rvert)$.
\end{claim}

\begin{proof}
We shall prove that $\lvert U_i \rvert \leq h_{i-1}(\lvert Y_1 \rvert)$ by induction on $i\geq 0$.
When $i=0$, we have $\lvert U_0 \rvert = \lvert Y_1 \rvert = h_{-1}(\lvert Y_1 \rvert)$.
Now assume that $i \geq 1$ and the claim holds for all smaller $i$.
By induction and \cref{BoundedGrowth}, 
\begin{align*}
\lvert U_i \rvert 
 = \lvert U_{i-1} \rvert + |\ngs{U_{i-1}}|
  \leq \lvert U_{i-1} \rvert+f(\lvert U_{i-1} \rvert) 
 = h_0(\lvert U_{i-1} \rvert) 
& \leq h_0(h_{i-2}(\lvert Y_1 \rvert)) \\
& = h_{i-1}(\lvert Y_1 \rvert).
\end{align*}
The case $i=|Y_1|+1$  proves the claim.
\end{proof}

Since Statement 2 does not hold, $\lvert Y_1 \rvert \leq k$.
By \cref{EnlargeClaim1}, $\lvert Y_1^* \rvert \leq h_{k}(k) \leq \eta$.
Recall that we proved that $\ngs{U_0} \neq \emptyset$.
So $\lvert Y_1^* \rvert > \lvert Y_1 \rvert$.
And by \cref{progress basic}, $\{v \in Y_1: F \cap L(v)\neq \emptyset\} = \{v \in Y_1^*: F \cap L^*(v) \neq \emptyset\}$.

\begin{claim}\label{EnlargeClaim2}
For every $L^*$-coloring $c$ of $G$, every monochromatic component with respect to $c$ intersecting $Y_1$ is contained in $Y_1^*$.
\end{claim}

\begin{proof}
For $i\in[|Y_1|]$, let $M_i$ be the monochromatic component with respect to $c$ containing $y_i$.
We shall prove that $V(M_i) \subseteq U_{i-1}$ for every $i\in[\lvert Y_1 \rvert]$.

For $i\in[|Y_1|]$, note that $U_i-U_{i-1}= \ngs{U_{i-1}}$.
Since $L_i$ is an $( U_i-U_{i-1}, \{\ell_i\} \cup F)$-progress, $\ell_i \not \in L_i(u)$ for every $u \in U_i-U_{i-1}$, and for every $v \in \nls{U_{i-1}}$, 
either $v \not \in N_G(V(M_i) \cap U_{i-1})$ or $\ell_i \not \in L_i(v)$.
Since $L_i(v) \supseteq L^*(v)$ for every $v \in V(G)$, we have $V(M_i) \cap N_G(U_{i-1})=\emptyset$ since $M_i$ is connected.
Since $M_i$ is connected and $V(M_i) \cap U_{i-1} \supseteq V(M_i) \cap U_0 \neq \emptyset$, we have $V(M_i) \subseteq U_{i-1}$.

Every monochromatic component $M$ with respect to $c$ intersecting $Y_1$ equals $M_j$ for some $j \in [\lvert Y_1 \rvert]$, 
so $V(M) \subseteq U_{j-1} \subseteq Y_1^*$.
\end{proof}

Since $L^*$ is an $(\ngs{U_{\lvert Y_1 \rvert}},\{\ell\} \cup F)$-progress of $L_{\lvert Y_1 \rvert}$, by \cref{progress basic}, for every $x \in F \cup \{\ell\}$ and $y \in Y_1^*$ with $x \in L^*(y)$, we have $\{v \in N_G(y) - Y_1^*: x \in L^*(v)\}=\emptyset$.
Since Statement 3 does not hold, by \cref{EnlargeClaim2}, 
there exists an $(\eta,g)$-bounded $L^*$-coloring $c^*$ of $G$ such that for every $x \in F$, the set of vertices colored $x$ is a stable set.
So every monochromatic component with respect to $c^*$ contains at most $\eta^2 g(\eta)$ vertices.
By \cref{EnlargeClaim1,EnlargeClaim2}, the union of the monochromatic components with respect to $c^*$ intersecting $Y_1$ contains at most $\lvert Y_1^* \rvert \leq h_{\lvert Y_1 \rvert}(\lvert Y_1 \rvert) \leq g(\lvert Y_1 \rvert)$ vertices.
Therefore, $c^*$ is an $(\eta,g)$-bounded $L$-coloring of $G$ such that for every $x \in F$, the set of vertices colored $x$ is a stable set. So Statement 1 holds. This contradiction proves the lemma.
\end{proof}

\subsection{Layered Treewidth}

We now take an excursion to introduce a tool used in our main proofs. 
A \emph{layering} of  graph $G$ is a partition $(V_1,\dots,V_n)$ of $V(G)$ such that for each edge $vw\in E(G)$ there exists $i\in[n-1]$ such that $v,w\in V_i\cup V_{i+1}$. The \emph{layered treewidth} of a graph $G$ is the minimum integer $\ell$ such that $G$ has a tree decomposition $(T,\mathcal{X}=(X_x:x\in V(T)))$ and a layering $(V_1,\dots,V_n)$, such that $|X_x\cap V_i|\leq\ell$ for each $x\in V(T)$ and $i\in[n]$. 

Layered treewidth was introduced by \citet{DMW17}. They proved that every planar graph has layered treewidth at most 3;  more generally, that every graph with Euler genus at most $g$ has layered treewidth at most $2g+3$; and most generally, that a minor-closed class has bounded layered treewidth if and only if it excludes some apex graph as a minor. Several interesting non-minor-closed classes also have bounded layered treewidth \citep{DEW17,BDDEW18,DJMNW18}.

In our companion paper~\cite{LW1}, we prove that graphs of bounded layered treewidth and with no $K_{s,t}$ subgraph are $(s+2)$-colorable with bounded clustering. In fact, we prove the following stronger result.

\begin{theorem}[\citep{LW1}]
\label{ltwmain}
For all $s,t,w,\xi\in\mathbb{N}$ there exists $\eta\in\mathbb{N}$ such that if $G$ is a graph with no $K_{s,t}$ subgraph such that $G-Z$ has layered treewidth at most $w$ for some $Z \subseteq V(G)$ with $\lvert Z \rvert \leq \xi$, then $G$ is $(s+2)$-colorable with clustering $\eta$.
\end{theorem}

We actually need the following more precise result that fits our list coloring setup. 
Let $G$ be a graph and $Z\subseteq V(G)$. A {\it $Z$-layering} $\V$ of $G$ is an ordered partition $(V_1,V_2,\dots)$ of $V(G)-Z$ such that for every edge $e$ of $G-Z$, there exists $i\in\mathbb{N}$ such that both endpoints of $e$ are contained in $V_i \cup V_{i+1}$. For a tree decomposition $(T,\X)$ of $G$, the {\it $\V$-width} is 
\begin{equation*}
\max_{i\in\mathbb{N}} \max_{t \in V(T)}\lvert X_t \cap V_i \rvert.\end{equation*}
Let $s\in\mathbb{N}$. 
A list-assignment $L$ of $G$ is {\it $(s,\V)$-compatible} if:
\begin{itemize}
	\item $L(v) \subseteq [s+2]$ for every $v \in V(G)$, and
	\item $i \not \in L(v)$ for every $i \in [s+2]$ and $v \in \bigcup (V_j: j \equiv i \pmod{s+2})$.
\end{itemize}
For $Y_1\subseteq V(G)$, we say that $(Y_1,L)$ is a {\it $\V$-standard pair} if $L$ is an $(s,1,Y_1)$-list-assignment and is $(s,\V)$-compatible. 

\begin{theorem}[\cite{LW1}] \label{bounded layered tw}
For all $s,t,w,k,\xi\in\mathbb{N}$, there exists $\eta^*\in\mathbb{N}$ such that if $G$ is a graph with no $K_{s.t}$ subgraph, $Z$ is a subset of $V(G)$ with $\lvert Z \rvert \leq \xi$, $\V$ is a $Z$-layering of $G$, $(T,\X)$ is a tree decomposition of $G-Z$ with $\V$-width at most $w$, $Y_1$ is a subset of $V(G)$ with $\lvert Y_1 \rvert \leq k$, $L$ is an $(s,\V)$-compatible list-assignment of $G$ such that $(Y_1,L)$ is a $\V$-standard pair, then there exists an $L$-coloring of $G$ with clustering $\eta^*$. 
\end{theorem}

\section{Tangles and Treewidth}
\label{TanglesTreewidth}

The goal of this section is to prove \cref{twmain} regarding graphs of bounded treewidth and to set-up machinery for the proofs of \cref{minorbasic,oddminorbasic} in subsequent sections. 

\begin{lemma} 
\label{coloring or tangle}
For all $s,t,\theta,\eta,r\in\mathbb{N}$ with $\eta \geq 9\theta+1$, for every nondecreasing function $g$ with domain ${\mathbb N_0}$, if $G$ is a graph with no $K_{s,t}$ subgraph, $Y_1$ is a subset of $V(G)$ with $9\theta+1 \leq \lvert Y_1 \rvert \leq \eta$, $F$ is a set of colors with $\lvert F \rvert \leq r$, and $L$ is an $(s,r,Y_1)$-list-assignment of $G$ such that $\{y \in Y_1: x \in L(y)\}$ is a stable set in $G$ for every $x \in F$, then at least one of the following holds:
	\begin{enumerate}
		\item There exists an $(\eta,g)$-bounded $L$-coloring of $G$ such that for every $x \in F$, the set of vertices colored $x$ is a stable set in $G$.
		\item There exist an induced subgraph $G'$ of $G$ with $\lvert V(G') \rvert < \lvert V(G) \rvert$, a subset $Y_1'$ of $V(G')$ with $\lvert Y_1' \rvert \leq \eta$, and an $(s,r,Y_1')$-list-assignment $L'$ of $G'$ such that the following hold:
		\begin{enumerate}
			\item $L'(v) \subseteq L(v)$ for every $v \in V(G')$.
			\item There does not exist an $(\eta,g)$-bounded $L'$-coloring of $G'$ such that for every $x \in F$, the set of vertices colored $x$ is a stable set in $G'$.
			\item $\{v \in Y_1: F \cap L(v) \neq \emptyset\} \cap V(G') = \{v \in Y_1': F \cap L'(v) \neq \emptyset\}$.
			\item For every $v \in V(G')-Y_1'$, we have $L'(v) \cap F = L(v) \cap F$.
		\end{enumerate} 
		\item $\T:=\{(A,B): \lvert V(A \cap B) \rvert < \theta, \lvert V(A) \cap Y_1 \rvert \leq 3\theta\}$ is a tangle of order $\theta$ in $G$.
	\end{enumerate}
\end{lemma}

\begin{proof}
Suppose that Statements 1, 2 and 3 do not hold. 
Since $\T$ is not a tangle, one of (T1), (T2) or (T3) is violated.

Suppose that (T2) violated.
So there exist $(A_i,B_i) \in \T$ for $i \in [3]$ such that $A_1 \cup A_2 \cup A_3=G$.
Hence $\lvert Y_1 \rvert \leq \sum_{i=1}^3 \lvert A_i \cap Y_1 \rvert \leq 9\theta$, a contradiction.
So $\T$ satisfies (T2).
Similarly, for every $(A,B) \in \T$, we have $V(A) \neq V(G)$; otherwise $\lvert Y_1 \rvert=\lvert V(A) \cap Y_1 \rvert \leq 3\theta$, a contradiction.
So $\T$ satisfies (T3).

Therefore, (T1) is violated.
So there exists a separation $(A,B)$ of $G$ of order less than $\theta$ such that $(A,B) \not \in \T$ and $(B,A) \not \in \T$.
That is, $\lvert V(A) \cap Y_1 \rvert > 3\theta$ and $\lvert V(B) \cap Y_1 \rvert > 3\theta$.
In particular, $V(A) \neq V(G) $ and $V(B)\neq V(G)$.

Let $Y_A:=(Y_1 \cap V(A)) \cup V(A \cap B)$ and $Y_B:=(Y_1 \cap V(B)) \cup V(A \cap B)$.
Note that $\max\{\lvert Y_A \rvert, \lvert Y_B \rvert\} \leq \lvert Y_1 \rvert-2\theta<\eta$.
Let $G_A:=G[V(A)]$ and $G_B:=G[V(B)]$.

Let $L'$ be a $(V(A \cap B),F)$-progress $L'$ of $L$.
Since $\lvert F \rvert \leq r$ and $L$ is an $(s,r,Y_1)$-list-assignment of $G$, by \cref{progress basic}, $L'$ is an $(s,r,Y_A \cup Y_B)$-list-assignment of $G$.
Let $L_A=L'|_{V(A)}$ and $L_B=L'|_{V(B)}$.
Since $V(A \cap B) \subseteq Y_A \cap Y_B$, $L_A$ is an $(s,r,Y_A)$-list-assignment of $G_A$, and $L_B$ is an $(s,r,Y_B)$-list-assignment of $G_B$.
Moreover, by \cref{progress basic}, for every $C \in \{A,B\}$, $L_C$ is an $(s,r,Y_C)$-list-assignment of $G_A$ such that 
	\begin{itemize}
		\item $L_C(v) \subseteq L(v)$ for every $v \in V(G_C)$.
		\item $\{v \in Y_1: F \cap L(v) \neq \emptyset\} \cap V(C) = \{v \in Y_C: F \cap L_C(v) \neq \emptyset\}$.
		\item For every $v \in V(G_C)-Y_C$, we have $L_C(v) \cap F = L(v) \cap F$.
	\end{itemize}

Since Statement 2 does not hold, for $C \in \{A,B\}$, there exists an $(\eta,g)$-bounded $L_C$-coloring $c_C$ of $G_C$ such that for every $x \in F$, the set of vertices colored $x$ is a stable set in $G_C$.
By construction, $c_A(v)=c_B(v)$ for every $v \in V(A \cap B)$.
Define $c(v):=c_A(v)$ of $v \in V(A)$ and define $c(v):=c_B(v)$ if $v \in V(B)$.
Clearly, $c$ is an $L$-coloring such that for every $x \in F$, the set of vertices colored $x$ is a stable set in $G$.

Let $C$ be the union of the monochromatic components with respect to $c$ intersecting $Y_1 \cup V(A \cap B)$.
Then $V(C) \cap Y_A \neq \emptyset $ and $V(C) \cap Y_B \neq \emptyset$.
By construction, 
\begin{align*}
\lvert V(C) \rvert 
& \leq \lvert V(C) \cap V(A) \rvert + \lvert V(C) \cap V(B) \rvert \\
& \leq \lvert Y_A \rvert^2 g(\lvert Y_A \rvert)+ \lvert Y_B \rvert^2 g(\lvert Y_B \rvert) \\
& \leq (\lvert Y_A \rvert^2 + \lvert Y_B \rvert^2) g(\lvert Y_1 \rvert) \\
& \leq ((\lvert Y_A \rvert + \lvert Y_B \rvert)^2-2\lvert Y_A \rvert \lvert Y_B \rvert)g(\lvert Y_1 \rvert).
\end{align*}
Since $\lvert Y_A \rvert \geq 3\theta$ and $\lvert Y_B \rvert \geq 3\theta$,  
\begin{equation*}\lvert Y_1 \rvert \theta+2\theta^2 \leq (\lvert Y_A \rvert + \lvert Y_B \rvert) \theta+2\theta^2 \leq \lvert Y_A \rvert \cdot \frac{\lvert Y_B \rvert}{3} + \lvert Y_B \rvert \cdot \frac{\lvert Y_A \rvert}{3} + 2 \cdot \frac{\lvert Y_A \rvert}{3} \cdot \frac{\lvert Y_B \rvert}{3} \leq \lvert Y_A \rvert\cdot \lvert Y_B \rvert.\end{equation*}
Therefore, 
\begin{align*}
\lvert V(C) \rvert 
& \leq ((\lvert Y_A \rvert + \lvert Y_B \rvert)^2-2\lvert Y_A \rvert \lvert Y_B \rvert)g(\lvert Y_1 \rvert) \\
& \leq ((\lvert Y_1 \rvert +\theta)^2-2\lvert Y_1 \rvert \theta - 4\theta^2) g(\lvert Y_1 \rvert) \\
& \leq \lvert Y_1 \rvert^2 g(\lvert Y_1 \rvert).
\end{align*}

Since Statement 1 does not hold, $c$ is not an $(\eta,g)$-bounded $L$-coloring of $G$.
Since $\lvert Y_1 \rvert^2 g(\lvert Y_1 \rvert) \leq \eta^2g(\eta)$, there exists a monochromatic component $M$ with respect to $c$ disjoint from $Y_1 \cup V(A \cap B)$ containing at least $\eta^2 g(\eta)+1$ vertices.
However, since $M$ is disjoint from $V(A \cap B)$, either $M$ is contained in $G_A$ or $M$ is contained in $G_B$.
So $M$ is a monochromatic component with respect to $c_A$ or $c_B$.
Since $c_A,c_B$ are $(\eta,g)$-bounded, $M$ contains at most $\eta^2 g(\eta)$ vertices, a contradiction.
This proves the lemma.
\end{proof}

\begin{theorem} \label{bdd tw Kst}
For all $s,t,w\in\mathbb{N}$, there exist $\eta\in\mathbb{N}$ and a nondecreasing function $g$ such that if $G$ is a graph of treewidth at most $w$ and with no $K_{s,t}$ subgraph, $Y_1$ is a subset of $V(G)$ with $\lvert Y_1 \rvert \leq \eta$, and $L$ is an $(s,1,Y_1)$-list-assignment of $G$, then there exists an $(\eta,g)$-bounded $L$-coloring of $G$.
\end{theorem}

\begin{proof}
Define $\eta$ and $g$ to be the number $\eta$ and the function $g$  in \cref{enlarge Y} by taking $s=s$, $t=t$, $k=9w+18$ and $r=1$.
Note that $g(x) \geq \eta > 9w+18$ for every $x \in {\mathbb N_0}$ by \cref{enlarge Y}.

Suppose to the contrary that this theorem does not hold.
So there exist a graph of treewidth at most $w$ and with no $K_{s,t}$ subgraph, a subset $Y_1$ of $V(G)$ with $\lvert Y_1 \rvert \leq \eta$, and an $(s,1,Y_1)$-list-assignment $L$ of $G$ such that there does not exist an $(\eta,g)$-bounded $L$-coloring of $G$.
We further assume that $\lvert V(G) \rvert$ is as small as possible and subject to this, $\lvert Y_1 \rvert$ is as large as possible.
Since $g(x) \geq \eta$ for every $x \in {\mathbb N_0}$, we have $\lvert V(G) \rvert > \eta$, as otherwise any $L$-coloring of $G$ is $(\eta,g)$-bounded.

By \cref{enlarge Y} and the choice of $G$ and $Y_1$, we have $\lvert Y_1 \rvert > 9w+18$.
By \cref{coloring or tangle}, there exists a tangle of order $w+2$ in $G$.
But $G$ has treewidth at most $w$, there exists no tangle of order $w+2$ in $G$ by \citep[Lemma~(5.2)]{RS-X}, a contradiction.
This proves the theorem.
\end{proof}

\begin{corollary}
\label{twcor}
For all $s,t,w\in\mathbb{N}$, there exists $\eta\in\mathbb{N}$ and a nondecreasing function $g$ with domain ${\mathbb N_0}$ such that if $G$ is a graph of treewidth at most $w$ and with no $K_{s,t}$ subgraph and $L$ is a $(s+1)$-list-assignment of $G$, then there exists an $L$-coloring of $G$ with clustering $\eta^2 g(\eta)$.
\end{corollary}

\begin{proof}
Define $L'(v)$ to be an $(s+1)$-element subset of $L(v)$ for every $v \in V(G)$.
Clearly, $L'$ is an $(s,1,\emptyset)$-list-assignment of $G$, and every $L'$-coloring is an $L$-coloring.
The result immediately follows from \cref{bdd tw Kst}.
\end{proof}

Observe that \cref{twcor} implies \cref{twmain}.

\section{Graph Minor Structure Theorem}

Our proofs of \cref{minorbasic,oddminorbasic} depend on the Graph Minor Structure Theorem, which we now introduce. Recall the definition of an $H$-minor $\alpha$ in a graph $G$ from \cref{Definitions}. A tangle~$\T$ in $G$ \emph{controls} an $H$-minor $\alpha$ if there does not exist $(A,B) \in \T$ of order less than $\abs{V(H)}$ such that $V(\alpha(h)) \subseteq V(A)$ for some $h\in V(H)$. We use the following theorem of \citet{GGRSV09} on odd minors.

\begin{theorem}[{\cite[Theorem 13]{GGRSV09}}]
\label{structure odd minor}
There is a constant $c$ such that for all $\ell\in\mathbb{N}$, if $t:=\lceil c \ell \sqrt{\log 12\ell}\rceil$ then for every graph $G$ that contains a $K_t$-minor $\alpha$, either $G$ contains an odd $K_\ell$-minor, or there exists a set $X$ of vertices with $\abs{X}<8\ell$ such that the (unique) block $U$ of $G-X$ that intersects all branch sets of $\alpha$ disjoint from $X$ is bipartite.
\end{theorem}

A \emph{society} is a pair $(S,\Omega)$, where $S$ is a graph and $\Omega$ is a cyclic permutation of a subset $\overline{\Omega}$ of~$V(S)$. For $\rho\in\mathbb{N}_0$, a society $(S,\Omega)$ is a \emph{$\rho$-vortex} if for all distinct $u,v \in \overline{\Omega}$, there do not exist $\rho+1$ disjoint paths in $S$ between $I \cup \{u\}$ and $J \cup \{v\}$, where $I$ is the set of vertices in $\overline{\Omega}$ after $u$ and before $v$ in the order $\Omega$, and $J$ is the set of vertices in $\overline{\Omega}$ after $v$ and before $u$.
For a society $(S,\Omega)$ with $\overline{\Omega} = (v_1,v_2,\ldots,v_n)$ in order, a \emph{vortical decomposition} of~$(S,\Omega)$ is a path decomposition $(t_1t_2\cdots t_n, \X )$ such that the $i$-th bag $X_i$ of~$\X$ contains the $i$-th vertex $v_i$ for each $i \in [n]$. The \emph{adhesion} of such a vortical decomposition is $\max\{ |X_i \cap X_{j} | : i,j\in[n], i\neq j\}$. We use the following theorem of \citet{RS-IX}.

\begin{theorem}[{\cite[(8.1)]{RS-IX}}] \label{vortical}
Every $\rho$-vortex has a vortical decomposition of adhesion at most $\rho$. 
\end{theorem}

A \emph{segregation} of a graph $G$ is a set $\Se$ of societies such that:
\begin{itemize}
	\item $S$ is a subgraph of~$G$ for every $(S, \Omega) \in \Se$, and $\bigcup \{S: (S,\Omega) \in \Se\}=G$, and
	\item for all distinct $(S,\Omega)$ and $(S', \Omega') \in \Se$, we have $V(S \cap S') \subseteq \overline{\Omega} \cap \overline{\Omega'}$ and $E(S \cap S') = \emptyset$.
\end{itemize}
We write $V(\Se) = \bigcup \{\overline{\Omega}: (S, \Omega) \in \Se\}$.
For positive integers $\kappa,\rho$, a segregation $\Se$ is of {\it type} $(\kappa,\rho)$ if there exist disjoint subsets $\Se_1,\Se_2$ of $\Se$ with $\Se=\Se_1 \cup \Se_2$ and $\lvert \Se_2 \rvert \leq \kappa$ such that $\lvert \overline{\Omega} \rvert \leq 3$ for every $(S,\Omega) \in \Se_1$, and every member of $\Se_2$ is a $\rho$-vortex.
For a tangle $\T$ in $G$, a segregation $\Se$ of $G$ is \emph{$\T$-central} if for every $(S,\Omega) \in \Se$, there exists no $(A,B) \in \T$ with $B \subseteq S$.

Let $\Sigma$ be a surface. 
For every subset $\Delta$ of $\Sigma$, we denote the closure of $\Delta$ by $\overline{\Delta}$ and the boundary of~$\Delta$ by $\partial\Delta$.
An \emph{arrangement} of a segregation $\Se = \{(S_1, \Omega_1), \ldots, (S_k, \Omega_k)\}$ in $\Sigma$ is a function $\alpha$ with domain $\Se \cup V(\Se)$, such that:
\begin{itemize}
	\item For $i \in [k]$, $\alpha(S_i, \Omega_i)$ is a closed disk $\Delta_i \subseteq \Sigma$, and $\alpha(x) \in \partial\Delta_i$ for each $x \in \overline{\Omega_i}$.
	\item For $i,j\in[k]$ with $i \neq j$, if $x \in \Delta_i \cap \Delta_j$, then $x=\alpha(v)$ for some $v \in \overline{\Omega_i} \cap \overline{\Omega_j}$.
	\item For all distinct $x,y \in V(\Se)$, we have $\alpha(x) \neq \alpha(y)$.
	\item For $i\in[k]$, $\Omega_i$ is mapped by $\alpha$ to a natural order of~$\alpha(\overline{\Omega_i})$ determined by $\partial\Delta_i$.
\end{itemize}
An arrangement is \emph{proper} if $\Delta_i \cap \Delta_j = \emptyset$ whenever 
$\abs{ \overline{\Omega_i} }, \abs{ \overline{\Omega_j}} >3$,
for all $1 \leq i < j \leq k$.

For a tangle $\T$ in a graph $G$ of order $\theta$ and a subset $Z$ of $V(G)$ with $\lvert Z \rvert < \theta$, $\T-Z$ is defined to be the set of all separations $(A',B')$ of $G-Z$ of order less than $\theta-\lvert Z \rvert$ such that there exists $(A,B) \in \T$ with $Z \subseteq V(A \cap B)$, $A'=A-Z$ and $B'=B-Z$.
It is proved in \citet{RS-X} that $\T-Z$ is a tangle in $G-Z$ of order $\theta-\lvert Z \rvert$.

The following is the Graph Minor Structure Theorem of \citet{RS-XVI}. 

\begin{theorem}[{\cite[(3.1)]{RS-XVI}}] \label{structure minor free}
For every graph $H$, there exist $\kappa,\rho,\xi,\theta \in\mathbb{N}$ such that if $\T$ is a tangle of order at least $\theta$ in a graph $G$ controlling no $H$-minor of $G$, then there exist $Z \subseteq V(G)$ with $\lvert Z \rvert \leq \xi$ and a $(\T-Z)$-central segregation $\Se$ of $G-Z$ of type $(\kappa,\rho)$ such that $\Se$ has a proper arrangement in some surface in which $H$ cannot be embedded.
\end{theorem}

\section{Main Proofs}
\label{MainProofs}

The goal of this section is to prove \cref{minorbasic,oddminorbasic}. Let $G$ be a graph.
A {\it location} $\L$ in $G$ is a collection of separations of $G$ such that $A \subseteq B'$ for every ordered pair of distinct members $(A,B),(A',B')$ of $\L$.
Define $G(\L)$ to be the graph $G[\bigcap_{(A,B) \in \L}V(B)]$.

Let $s \in {\mathbb N}$, $r \in {\mathbb N_0}$ and $\ell \in [0,s+2]$.
Let $G$ be a graph and let $Y_1\subseteq V(G)$.
A list-assignment $L$ of $G$ is an {\it $(s,Y_1,\ell,r)$-list-assignment} if the following hold:
\begin{itemize}
	\item[(R1)] $L(v) \subseteq [s+2+r]$ for all $v\in V(G)$.
	\item[(R2)] $L$ is an $(s,r+2,Y_1)$-list-assignment of $G$.
	\item[(R3)] For every $y \in Y_1$ and color $x \in \{\ell\} \cup [s+3,s+2+r]$ with $x \in L(y)$, we have $\{v \in N_G(y)-Y_1: x \in L(v)\}=\emptyset$.
	\item[(R4)] For every $x \in [s+3,s+2+r]$, the set $\{y \in Y_1: x \in L(y)\}$ is a stable set in $G$. 
	\item[(R5)] For every $v \in V(G)-Y_1$, 
	\begin{equation*}\lvert \{y \in N_G(v) \cap Y_1: L(y) \subseteq [s+3,s+r+2]\} \rvert = r-\lvert L(v) \cap [s+3,s+2+r]\rvert.\end{equation*}
\end{itemize}
Note that every $(s,Y_1,0,0)$-list-assignment of $G$ is an $(s,2,Y_1)$-list-assignment. For graph minors we use $\ell=r=0$. The more general setting is used for odd minors. 

Let $G$ be a graph and $Z,Y_1 \subseteq V(G)$.
Let $s\in {\mathbb N}$, $r \in {\mathbb N_0}$ and $\ell \in [0,s+2]$.
Let $L$ be an $(s,Y_1,\ell,r)$-list-assignment of $G$.
A {\it $(Z,\ell)$-growth} of $L$ is a list-assignment $L'$ of $G$ defined as follows:
\begin{itemize}
	\item Let $Y_1^{(-1)}:=Y_1$ and $L^{(-1)}(v):=L(v)$ for every $v \in V(G)$. 
	\item Let $U_0:=Z$, and for each $i \geq 1$, let $U_i := \ngs{Y_1^{(i-1)}}$. 
	\item Let $Y^{(i)}_1 :=Y^{(i-1)}_1 \cup U_i$ for $i \geq 0$.
	\item For every $i \geq 0$, let $L^{(i)}$ be a $(U_i,\{\ell,i\} \cup [s+3,s+2+r])$-progress of $L^{(i-1)}$ such that $L^{(i)}(v) \cap [s+3,s+2+r] = L^{(i-1)}(v) \cap [s+3,s+2+r]$ for every $v \in V(G)-(Y_1^{(i-1)} \cup U_i)$. 
		(Note that $\lvert \{\ell,i\} \cup [s+3,s+2+r] \rvert \leq r+2$, so a $(U_i,\{\ell,i\} \cup [s+3,s+2+r])$-progress of $L^{(i-1)}$ can be defined, by (R2); furthermore, such an $L^{(i)}$ exists since $s+2 \geq s-1$ and for every $v \in V(G)-(Y_1^{(i-1)} \cup U_i)$, $L^{(i)}(v)$ can be obtained from $L^{(i-1)}(v)$ by removing elements, and $L^{(i)}(v) \neq L^{(i-1)}(v)$ only when $v \in \nls{Y_1^{(i-1)} \cup U_i}$.)
	\item Let $Y_1' := Y^{(s+2)}_1$ and $L' := L^{(s+2)}$.
\end{itemize}

\begin{lemma} 
\label{Z growth}
For all $s,t\in\mathbb{N}$ there exists a function $h:\mathbb{N}_0\to\mathbb{N}_0$ such that if $G$ is a graph with no $K_{s,t}$ subgraph, $Z,Y_1 \subseteq V(G)$, $\ell \in [0,s+2]$, $r \in {\mathbb N_0}$, and $L$ is an $(s,Y_1,\ell,r)$-list-assignment of $G$, then for each $(Z,\ell)$-growth $L'$ of $L$ and set $Y_1' $ defined above  the following hold:
	\begin{enumerate}
		\item $L'$ is an $(s,Y_1',\ell,r)$-list-assignment.
		\item $L'(v) \subseteq L(v)$ for every $v \in V(G)$.
		\item $\lvert Y_1' \rvert \leq h(\lvert Y_1 \cup Z \rvert)$.
		\item For every $L'$-coloring $c'$ of $G$, every monochromatic component $M$ with respect to $c'$ intersecting $Y_1 \cup Z$ is contained in $G[Y_1']$.
		Furthermore, if the vertices of $M$ are assigned colors in $\{\ell\} \cup [s+3,s+2+r]$, then $M \subseteq G[Y_1]$.
		\item If $(A,B)$ is a separation of $G$ with $Z \subseteq V(A \cap B)$, then $\lvert Y_1' \cap V(A) \rvert \leq h(\lvert V(A \cap B) \rvert + \lvert Y_1 \cap V(A) \rvert)$ for every $(A,B) \in \L$. 
	\end{enumerate}
\end{lemma}

\begin{proof}
Let $f:\mathbb{N}_0\to\mathbb{N}_0$ be the function $f_{s,t}$ in \cref{BoundedGrowth}. 
Let $f_0:\mathbb{N}_0\to\mathbb{N}_0$ be the function defined by $f_0(x):=x+f(x)$ for $x\in \mathbb{N}_0$. 
For $i \geq 1$, let $f_i:\mathbb{N}_0\to\mathbb{N}_0$ be the function defined by $f_i(x):=x+f_{i-1}(x)+f(f_{i-1}(x))$.
Define $h:=f_{s+2}$.

Since $L$ is an $(s,r+2,Y_1)$-list-assignment and $L'$ is obtained from $L$ by repeatedly taking $(W_i,F_i)$-progress for some sets $W_i,F_i$ with $\lvert F_i \rvert \leq r+2$ and $\{\ell\} \cup [s+3,s+2+r] \subseteq F_i$ for all $i$, such that $L'(v) \cap [s+3,s+2+r] = L(v) \cap [s+3,s+2+r]$ for every $v \in V(G)-Y_1'$, we know that $L'$ is an $(s,r+2,Y_1')$-list-assignment of $G$ satisfiying (R5) such that 
\begin{equation*}\{v \in Y_1': L'(v) \cap (\{\ell\} \cup [s+3,s+2+r]) \neq \emptyset\} = \{v \in Y_1: L(v) \cap (\{\ell\} \cup [s+3,s+2+r]) \neq \emptyset\}\end{equation*} 
by \cref{progress basic}. 
So $L'$ satisfies (R1)--(R5), and Statement 2 holds.  Hence Statements 1 and 2 hold.

For every $i \geq 0$, let $U_i, Y_1^{(i)},L^{(i)}$ be the sets and list-assignment mentioned in the definition of a $(Z,\ell)$-growth.
By \cref{BoundedGrowth}, 
it is easy to verify that $\lvert U_i \rvert \leq f(\lvert Y^{(i-1)}_1 \rvert)$ and $\lvert Y^{(i)}_1 \rvert \leq \lvert Y^{(i-1)}_1 \rvert + \lvert U_i \rvert \leq f_i(\lvert Y_1 \cup Z \rvert)$ for every $i \geq 0$ by induction on $i$.
This proves Statement 3.

Let $c'$ be an $L'$-coloring of $G$.
Let $M_i$ be a monochromatic component with respect to $c'$ intersecting $Y_1 \cup Z$ such that all vertices of $M_i$ are colored $i$ for some $i \in [s+r+2]$.
If $i \in \{\ell\} \cup [s+3,s+2+r]$, then $V(M_i) \cap Z-Y_1 = \emptyset$ and $V(M_i) \subseteq Y_1$ since $c'$ is an $L$-coloring and $L$ satisfies (R3).
So we may assume that $i \in [s+2]-\{\ell\}$.
Since $L^{(i)}$ is a $(U_i,\{\ell,i\} \cup [s+3,s+2+r])$-progress of $L^{(i-1)}$, we have $V(M_i) \cap U_i = \emptyset$.
Since $M_i$ is connected and $V(M_i) \cap U_i = \emptyset$ and $V(M_i) \cap (Y_1 \cup Z) \neq \emptyset$, we have either $V(M_i) \subseteq Y_1^{(i-1)}$ or there exists $xy \in E(M_i)$ such that $y \in Y_1^{(i-1)}$ and $x \in N_G(y) \cap \nls{Y_1^{(i-1)}}$.
But the latter is impossible by (L3). 
Hence $V(M_i) \subseteq Y_1^{(i-1)} \subseteq Y_1^{(s+1)} \subseteq Y_1'$.
Therefore, Statement 4 holds.

Let $(A,B)$ be a separation of $G$ with $Z \subseteq V(A \cap B)$.
For every $i \geq 0$,  since $(A,B)$ is a separation,
\begin{equation*}
U_i \cap V(A) \subseteq V(A \cap B) \cup \{v \in V(A)-Y_1^{(i-1)}: \lvert N_G(v) \cap Y_1^{(i-1)} \cap V(A) \rvert \geq s\}.
\end{equation*} 
Since $Z \subseteq V(A \cap B)$, by \cref{BoundedGrowth}, it is easy to prove by induction on $i$ that $\lvert U_i \cap V(A) \rvert \leq \lvert V(A \cap B) \rvert + f(\lvert Y_1^{(i-1)} \cap V(A) \rvert)$ and $\lvert Y_1^{(i)} \cap V(A) \rvert \leq \lvert Y_1^{(i-1)} \cap V(A) \rvert + \lvert U_i \cap V(A) \rvert \leq f_i(\lvert V(A \cap B) \rvert + \lvert Y_1 \cap V(A) \rvert)$.
 Statement 5 holds by taking $i=s+2$. 
\end{proof}

\begin{lemma} \label{NotControllingMinor}
For all $s,t,t' \in\mathbb{N}$, there exist $\theta^* \in {\mathbb N}$ and nondecreasing functions $g^*,\eta^*$ with domain ${\mathbb N_0}$ such that if $G$ is a graph with no $K_{s,t}$ subgraph, $\theta \in {\mathbb N}$ with $\theta \geq \theta^*$, $\eta \in {\mathbb N}$ with $\eta \geq \eta^*(\theta)$, $Y_1 \subseteq V(G)$ with $3\theta<\lvert Y_1 \rvert \leq \eta$, $\ell \in [0,s+2]$, $r \in {\mathbb N_0}$, $L$ is an $(s,Y_1,\ell,r)$-list-assignment of $G$, $g$ is a nondecreasing function with domain ${\mathbb N_0}$ and with $g \geq g^*$, and $\T$ is a tangle in $G$ of order $\theta$ that does not control a $K_{t'}$-minor, where $\T = \{(A,B): \lvert V(A \cap B) \rvert < \theta, \lvert V(A) \cap Y_1 \rvert \leq 3\theta\}$, then either:
	\begin{enumerate}
		\item there exists an $(\eta,g)$-bounded $L$-coloring of $G$ such that for every $x \in [s+3,s+2+r]$, the set of vertices colored $x$ is a stable set, or 
		\item there exist $(A^*,B^*) \in \T$, a set $Y_{A^*}$ with $\lvert Y_{A^*} \rvert \leq \eta^*(\theta)$ and $Y_1 \cap V(A^*) \subseteq Y_{A^*} \subseteq V(A^*)$, and an $(s,Y_{A^*},\ell,r)$-list-assignment $L_{A^*}$ of $G[V(A^*)]$ such that there exists no $(\eta,g)$-bounded $L_{A^*}$-coloring of $G[V(A^*)]$ such that for every $x \in [s+3,s+2+r]$, the set of vertices colored $x$ is a stable set.
	\end{enumerate}
\end{lemma}

\begin{proof}
Define the following:
\begin{itemize}
	\item Let $f$ be the function $f_{s,t}$ in \cref{BoundedGrowth}.
	\item Let $f_0:\mathbb{N}_0\to\mathbb{N}_0$ be the identity function. For every $i\in\mathbb{N}$, let $f_i:\mathbb{N}_0\to \mathbb{N}_0$ be the function defined by $f_i(x):=f_{i-1}(x)+f(f_{i-1}(x))$.
	\item Let $\kappa_0,\rho_0,\xi_0,\theta_0$ be the integers $\kappa,\rho,\xi,\theta$  in \cref{structure minor free} taking $H=K_{t'}$.
	\item Let $h:\mathbb{N}_0\to\mathbb{N}_0$ be the function in \cref{Z growth} taking $s=s$ and $t=t$.
	\item Let $\theta^*:= \theta_0+2\rho_0+\xi_0+3$.
	\item Let $\eta^*:\mathbb{N}_0\to\mathbb{N}_0$  be the function defined by $\eta^*(x):= h(f_3(h(4x)+x))$ for every $x \in {\mathbb N_0}$. 
	\item Let $\sigma$ be the maximum Euler genus of a surface in which $K_{t'}$ cannot be embedded.
	\item For every $x \in {\mathbb N_0}$, let $w_0(x):=(3\kappa_0(2\sigma+4)(2\rho_0+1)+1) \cdot (f_{1}(h(4x)+x)+1)$.
	\item For every $x \in {\mathbb N_0}$, let $\eta_1(x)$ be the number $\eta^*$ in \cref{bounded layered tw} taking $s=s$, $t=t+f_1(h(4x)+x)+(s+2)^2$, $w=(s+2)^2 \cdot w_0(x)$, $k=h(x+\xi_0)+s+2$ and $\xi=h(x+\xi_0)+s+2$.
	\item Let $g^*:\mathbb{N}_0\to\mathbb{N}_0$ be the function defined by $g^*(x):= \eta_1(2x) \cdot (1+f_1(h(4x+\xi_0)+x+\xi_0) \cdot \eta^*(x))$ for every $x \in {\mathbb N_0}$.  
\end{itemize}

Let $G$ be a graph with no $K_{s,t}$ subgraph, $\theta \in {\mathbb N}$ with $\theta \geq \theta^*$, $\eta \in {\mathbb N}$ with $\eta \geq \eta^*(\theta)$, $Y_1$ a subset of $V(G)$ with $3\theta< \lvert Y_1 \rvert \leq \eta$, $\ell \in [0,s+2]$, $r \in {\mathbb N_0}$, $g$ a nondecreasing function with domain ${\mathbb N}_0$ and with $g \geq g^*$, and $L$ an $(s,Y_1,\ell,r)$-list-assignment of $G$.
Let $\T = \{(A,B): \lvert V(A \cap B) \rvert < \theta, \lvert V(A) \cap Y_1 \rvert \leq 3\theta\}$.
Assume that $\T$ is a tangle in $G$ of order $\theta$ that does not control a $K_{t'}$-minor.

Suppose that there exists no $(\eta,g)$-bounded $L$-coloring of $G$ such that for every $x \in [s+3,s+2+r]$, the set of vertices colored $x$ is a stable set, and suppose that for every separation $(A,B) \in \T$, every set $Y_{A}$ with $\lvert Y_A \rvert \leq \eta^*(\theta)$ and $Y_1 \cap V(A) \subseteq Y_A \subseteq V(A)$ and every $(s,Y_A,\ell,r)$-list-assignment $L_A$ of $G[V(A)]$, there exists an $(\eta,g)$-bounded $L_A$-coloring of $G[V(A)]$ such that for every $x \in [s+3,s+2+r]$, the set of vertices colored $x$ is a stable set.

Since $\T$ does not control a $K_{t'}$-minor, by \cref{structure minor free}, there exist $Z \subseteq V(G)$ with $\lvert Z \rvert \leq \xi_0$, a $(\T-Z)$-central segregation $\Se$ of $G-Z$ of type $(\kappa_0,\rho_0)$, and a proper arrangement of $\Se$ in a surface $\Sigma$ in which $K_{t'}$ cannot be embedded. Let $\Se_1 := \{(S,\Omega) \in \Se: \lvert \overline{\Omega} \rvert \leq 3\}$, and let $\Se_2:= \Se-\Se_1$. 
Since $\Se$ is of type $(\kappa_0,\rho_0)$, $\lvert \Se_2 \rvert \leq \kappa_0$ and every member of $\Se_2$ is a $\rho_0$-vortex.

For each $(S,\Omega) \in \Se_1$, let $(A_S,B_S)$ be the separation of $G$ such that $S \subseteq A_S$, $V(A_S) = V(S) \cup Z$, $V(A_S \cap B_S)=\overline{\Omega} \cup Z$, and subject to these conditions, $\lvert E(A_S) \rvert$ is minimal.
Let $\L_1:=\{(A_S,B_S):(S,\Omega) \in \Se_1\}$.
By \cref{vortical}, for each $(S,\Omega) \in \Se_2$, there exists a vortical decomposition $(P_S,\X_S)$ of $(S,\Omega)$ of adhesion at most $\rho_0$.
For each $(S,\Omega) \in \Se_2$ and each bag $X$ of $(P_S,\X_S)$, let 
$$\partial X := (X \cap \overline{\Omega}) \cup \{v \in X: v \text{ belongs to a bag of }(P_S,\X_S) \text{ other than }X\}.$$
Let $(A_{S,X},B_{S,X})$ be the separation of $G$ such that $V(A_{S,X}):=X \cup Z$ and $V(A_{S,X} \cap B_{S,X}) := \partial X \cup Z$, and subject to these conditions, $E(A_{S,X})$ is minimal.
Let $\L_2:=\{(A_{S,X},B_{S,X}): (S,\Omega) \in \Se_2, X \in \X_S\}$ and  $\L:=\L_1 \cup \L_2$.
Note that every member of $\L$ has order at most $\xi_0+3+2\rho_0 < \theta^* \leq \theta$.
Since $\Se$ is $(\T-Z)$-central, $\L \subseteq \T$.
In addition, $\L$ is a location with $Z \subseteq V(A \cap B)$ for every $(A,B) \in \L$.

Let $L'$ be a $(Z,\ell)$-growth of $L$, and let $Y_1' :=\{v \in V(G): \lvert L'(v) \rvert=1\}$.
By \cref{Z growth}, $\lvert Y_1' \cap V(A) \rvert \leq h(\lvert V(A \cap B) \rvert + \lvert Y_1 \cap V(A) \rvert) \leq h(4\theta)$ for every $(A,B) \in \L$.
For each $(A,B) \in \L$, define:
$$U_A^{(0)}:=(Y_1' \cap V(A)) \cup V(A \cap B),\quad 
U_A^{(1)} := \ngs[{G[V(A)]}]{U_A^{(0)}}, \quad\text{and}\quad
Z_A := U_A^{(0)} \cup U_A^{(1)}.$$
Note that for every $(A,B) \in \L$, $\lvert Z_A \rvert \leq f_{1}(h(4\theta)+\theta)$.

Let $G'=G[\bigcup_{(A,B) \in \L}Z_A]$.
Since $Y_1' \cap V(A) \subseteq Z_A$ for every $(A,B) \in \L$, we have $Y_1' \subseteq V(G')$.

\begin{claim}\label{NotControllingMinorClaim1}
	There exists a $Y_1'$-layering $\V=(V_1,V_2,\dots,V_{\lvert \V \rvert})$ of $G'$ and a tree decomposition $(T,\X)$ of $G'$ with $\V$-width at most $w_0(\theta)$ such that for every $(A,B) \in \L$, there exist $\epsilon_A \in \{1,-1\}$ and $i_A\in\mathbb{N}$ such that:
	\begin{itemize}
		\item if $(A,B) \in \L_1$, then $V(A \cap B)-Y_1' \subseteq V_{i_A} \cup V_{i_A+\epsilon_A}$ and $\lvert V(A \cap B) \cap V_{i_A+\epsilon_A} \rvert \leq 1$,
		\item if $(A,B) \in \L_2$, then $V(A \cap B)-Y_1' \subseteq V_{i_A}$,
		\item $U_A^{(1)} \subseteq V_{i_A+\epsilon_A}$, and
		\item some bag of $(T,\X)$ contains $U_A^{(1)}$. 
	\end{itemize}
\end{claim}

\begin{proof}
Let $G_0$ be the graph obtained from $G[\bigcup_{(S,\Omega) \in \Se} \overline{\Omega}]$ by for each $(S,\Omega) \in\Se$, adding a new vertex $v_S$ adjacent to all vertices in $\overline{\Omega}$ and new edges such that there exists a cycle $C_S$ with $V(C_S)=\overline{\Omega}$ passing through its vertices in the order $\Omega$.
Since there exists a proper arrangement of $\Se$ in $\Sigma$, $G_0$ can be embedded in $\Sigma$.
Since $K_{t'}$ cannot be embedded in $\Sigma$, the Euler genus of $\Sigma$ is at most $\sigma$.
By \cite[Theorem~12]{DMW17}, there exists a layering $\V_0=(V_{0,1},V_{0,2},\dots,V_{0,\lvert \V_0 \rvert})$ of $G_0$ and a tree decomposition $(T_0,\X_0)$ of $G_0$ with $\V_0$-width at most $2\sigma+3$.
For each $p \in V(T_0)$, let $X_{0,p}$ be the bag at $p$ in $(T_0,\X_0)$.

For every $(S,\Omega) \in \Se_1$, since $\{v_S\} \cup \overline{\Omega}$ forms a clique in $G_0$, there exists $p_S \in V(T_0)$ such that $X_{0,p_S}$ contains this clique.
Let $T_0'$ be the tree obtained from $T_0$ by for each $(S,\Omega) \in \Se_1$, adding a new node $t_S$ and a new edge $t_Sp_S$.
For each $p \in V(T_0)$, let $X'_{0,p}:=X_{0,p}-\{v_S: (S,\Omega) \in \Se_1\}$.
For each $p \in V(T_0')-V(T_0)$, we know that $p=t_S$ for some $(S,\Omega) \in \Se_1$; define $X'_{0,p}:=\{v_S\} \cup \overline{\Omega}$.
Since $N_{G_0}(v_S)=\overline{\Omega}$ for each $(S,\Omega) \in \Se_1$, $(T_0',\X'_0)$ is a tree decomposition of $G_0$ of $\V_0$-width at most $2\sigma+3$, where $\X'_0=(X_{0,p}': p \in V(T_0'))$.

For each $(S,\Omega) \in \Se_1$, let $i_S$ be the index such that $v_S \in V_{0,i_S}$; since $\{v_S\} \cup \overline{\Omega}$ forms a clique in $G_0$, there exists $\epsilon_S \in \{-1,1\}$ such that $\overline{\Omega} \subseteq V_{0,i_S} \cup V_{0,i_S+\epsilon_S}$.
Let $\V_0':=(V_{0,1}',V_{0,2}',\dots,V_{0,\lvert \V_0 \rvert}')$ be the layering of $G_0$ obtained from $\V_0$ by moving $v_S$ from $V_{0,i_S}$ to $V_{0,i_S+\epsilon_S}$ for each $(S,\Omega) \in \Se_1$ with $\lvert \overline{\Omega} \cap V_{i_S+\epsilon_S} \rvert > \lvert \overline{\Omega} \cap V_{i_S} \rvert$.
Since for each $(S,\Omega) \in \Se_1$, the only bag in $(T_0',\X_0')$ containing $v_S$ has size $\lvert \overline{\Omega} \rvert+1 \leq 4$, the $\V_0'$-width of $(T_0',\X_0')$ is most $\max\{2\sigma+3,4\} \leq 2\sigma+4$.

Let $G_1$ be the graph obtained from $G[\bigcap_{(A,B) \in \L} V(B)]-Z$ by adding a new vertex $v_A$ adjacent to all vertices in $V(A \cap B)-Z$ for each $(A,B) \in \L_1$, and adding a new vertex $v_S$ adjacent to all vertices in $\overline{\Omega}$ for each $(S,\Omega) \in \Se_2$.
Note that $V(G_0) \subseteq V(G_1)$.
For each $i \in [\lvert \V'_0 \rvert]$, let 
\begin{equation*}V_{1,i}:=V'_{0,i} \cup \bigcup_{(S,\Omega) \in \Se_2, v_S \in V'_{0,i}} (V(G_1) \cap V(S)-\overline{\Omega}).
\end{equation*}
Let $\V_1:=(V_{1,1},V_{1,2},\dots,V_{1,\lvert \V_0' \rvert})$.
Since $N_{G_1}(V(G_1) \cap V(S)-\overline{\Omega}) \subseteq \overline{\Omega} \subseteq N_{G_1}(v_S)$ for every $(S,\Omega) \in \Se_2$, 
we have that $\V_1$ is a layering of $G_1$.
Note that for each $(S,\Omega) \in \Se_2$ and $X \in \X_S$, $X$ is the $j$-th bag in $\X_S$ for some $j$, so $X$ contains the $j$-th vertex in $\Omega$; we denote the $j$-th vertex in $\Omega$ by $v_X$, so $v_X \in X$.
For each node $p \in V(T'_0)$, define 
\begin{equation*}
X_{1,p} :=X'_{0,p} \cup \bigcup_{(S,\Omega) \in \Se_2}\bigcup_{X \in \X_S}\bigcup_{v_X \in X'_{0,p}} (X \cap V(G_1)).
\end{equation*} 
Since $(T'_0,\X'_0)$ is a tree decomposition of $G_0$ and $(P_S,\X_S)$ is a vortical decomposition of $S$ for every $(S,\Omega) \in \Se_2$, $(T'_0,\X_1)$ is a tree decomposition of $G_1$ by the existence of $C_S$, where for every $p \in V(T'_0)$, the bag of $(T'_0,\X_1)$ at $p$ is $X_{1,p}$.
Since $(P_S,\X_S)$ has adhesion at most $\rho_0$ for each $(S,\Omega) \in \Se_2$, the $\V_1$-width of $(T'_0,\X_1)$ is at most $(2\sigma+4)(2\rho_0+1)$.

For each $(S,\Omega) \in \Se_2$, since $v_S$ is adjacent to all vertices in $\overline{\Omega}$, there exists $i_S \in [\lvert \V_1 \rvert]$ such that $v_S \in V_{1,i_S}$, and $\overline{\Omega}$ is contained in $V_{1,i_S-1} \cup V_{1,i_S} \cup V_{1,i_S+1}$.
Let $\V_1'$ be the layering of $G_1$ obtained from $\V_1$ by for each $(S,\Omega) \in \Se_2$, merging $V_{1.i_S-1} \cup V_{1,i_S} \cup V_{1,i_S+1}$ into a layer.
Since $\lvert \Se_2 \rvert \leq \kappa_0$, each layer in $\V_1'$ is a union of at most $3\kappa_0$ layers in $\V_1$.
Hence the $\V_1'$-width of $(T_0',\X_1)$ is at most $3\kappa_0(2\sigma+4)(2\rho_0+1)$.
Denote $\V_1'=(V'_{1,1},V'_{1,2},...,V'_{1,\lvert \V_1'\rvert})$.

For each $(A,B) \in \L_1$, let $i_A$ be the index $i$ such that $v_A \in V'_{1,i}$; for each $(A,B) \in \L_2$, let $i_A$ be the index $i$ such that $v_S \in V'_{1,i}$, where $(S,\Omega)$ is the member in $\Se_2$ such that $V(A)-Z \subseteq V(S)$.
Note that by the definition of $\V_1'$, for each $(A,B) \in \L_2$, $V(A \cap B)-Z \subseteq V'_{1,i_A}$.
By adding empty layers, we may assume that $\lvert \V_1' \rvert \geq 2+\max_{(A,B) \in \L}i_A$ and $\min_{(A,B) \in \L}i_A \geq 3$.

For each $(A,B) \in \L_1$, since $\lvert V(A \cap B) -Z \rvert \leq 3$, we have that $(V(A \cap B)-Z) \cup \{v_A\}$ forms a clique in $G_1$; since $v_A \in V'_{1,i_A}$, there exists $\epsilon_A \in \{1,-1\}$ such that $V(A \cap B)-Z \subseteq V'_{1,i_A} \cup V'_{1,i_A+\epsilon_A}$.
Note that by the definition of $\V_0'$ and $\V_1'$, for each $(A,B) \in \L_1$, $\lvert V(A \cap B) \cap V'_{1,i_A} \rvert \geq \lvert V(A \cap B) \cap V'_{1,i_A+\epsilon_A} \rvert$, so $\lvert V(A \cap B) \cap V'_{1,i_A+\epsilon_A} \rvert \leq \lfloor \frac{\lvert V(A \cap B)-Z \rvert}{2} \rfloor \leq \lfloor \frac{3}{2} \rfloor= 1$.
For each $(A,B) \in \L_2$, define $\epsilon_A:=1$.

For each $j \in [\lvert \V_1' \rvert]$, define 
\begin{equation*}
	V_{2,j} := \big(V'_{1,j} \cup \bigcup_{(A,B) \in \L,i_A+\epsilon_A=j} U_A^{(1)} \big) \cap V(G').
\end{equation*}
Let $\V_2:=(V_{2,1},V_{2,2},\dots,V_{2,\lvert \V_1' \rvert})$.
For each $i \in [\lvert \V_2 \rvert]$, let $V_i:=V_{2,i}-Y_1'$.
Define $\V:=(V_1,V_2,...,V_{\lvert \V_2 \rvert})$.

Note that for each $(A,B) \in \L$, we have $V(A \cap B)-Y_1' \subseteq (V'_{1,i_A} \cup V'_{1,i_A+\epsilon_A})-Y_1' \subseteq (V_{2,i_A} \cup V_{2,i_A+\epsilon_A})-Y_1' \subseteq V_{i_A} \cup V_{i_A+\epsilon_A}$ and $\lvert V(A \cap B) \cap V_{i_A+\epsilon_A} \rvert \leq \lvert V(A \cap B) \cap V_{1,i_A+\epsilon_A}' \rvert \leq 1$. 
Moreover, if $(A,B) \in \L_2$, then $V(A \cap B)-Y_1' \subseteq V'_{1,i_A}-Y_1' \subseteq V_{2,i_A}-Y_1'=V_{i_A}$. 
For every $(A,B) \in \L$, since $U_A^{(1)} \subseteq V(A)-V(B)$, we have $N_{G'}(U_A^{(1)})-Y_1' \subseteq V(A \cap B)-Y_1' \subseteq (V_{i_A} \cup V_{i_A+\epsilon_A})-Y_1'$.
So $\V$ is a $Y_1'$-layering of $G'$ since $U_A^{(1)} \subseteq V_{i_A+\epsilon_A}$ for every $(A,B) \in \L$.

For each $p \in V(T'_0)$, define 
\begin{equation*}
X_{2,p}:= \big(X_{1,p} \cup \bigcup_{(A,B) \in \L_1, v_A \in X_{1,p}} \!\!\!\!\!\!\!\!\!\!\!\! (Z_A-V(G_1)) \quad
\cup 
\bigcup_{(A,B) \in \L_2,V(A \cap B)-Z \subseteq X_{1,p}, V(A \cap B)-Z \neq \emptyset} \!\!\!\!\!\!\!\!\!\!\!\!\!\!\!\!\!\!\!\!\!\! (Z_A-V(G_1)) \big) \quad \cap\quad  V(G').
\end{equation*}
Let $p_0$ be a node of $T_0'$.
Let $T$ be the tree obtained from $T_0'$ by adding, for each $(A,B) \in \L_2$ with $V(A \cap B)-Z=\emptyset$, a new node $t_A$ and a new edge $t_Ap_0$.
For each $p \in V(T_0')$, let $X_p:=X_{2,p} \cup Y_1'$; for each $p \in V(T)-V(T_0')$, we know $p=t_A$ for some $(A,B) \in \L_2$ with $V(A \cap B)-Z=\emptyset$; let $X_p:=(Z_A-V(G_1)) \cup Y_1'$. Let $\X:=(X_p: p \in V(T))$.

Note that for every $(A,B) \in \L_1$, we have $N_{G'}(Z_A-V(G_1)) \subseteq N_{G_1}(v_A) \cup Y_1'$; for every $(A,B) \in \L_2$, we have $N_{G'}(Z_A-V(G_1)) \subseteq V(A \cap B) \cup Y_1'$. So $(T,\X)$ is a tree decomposition of $G'$.
For every $(A,B) \in \L$, since $U_A^{(1)} \subseteq Z_A-V(G_1)$, some bag of $(T,\X)$ contains $U_A^{(1)}$.

Since $\V$ is a $Y_1'$-layering, and $\lvert Z_A \rvert \leq f_{1}(h(4\theta)+\theta)$ for every $(A,B) \in \L$, the $\V$-width of $(T,\X)$ is at most $(3\kappa_0(2\sigma+4)(2\rho_0+1)+1) \cdot (f_{1}(h(4\theta)+\theta)+1) \leq w_0(\theta)$.
This proves the claim.
\end{proof}

Let $\V:=(V_1,V_2,\dots,V_{\lvert \V \rvert})$ be the $Y_1'$-layering of $G'$ mentioned in \cref{NotControllingMinorClaim1}.
For every $(A,B) \in \L$, let $i_A$ and $\epsilon_A$ be the integers mentioned in \cref{NotControllingMinorClaim1}.

Let $G^*$ be the graph obtained from $G'$ by:
	\begin{itemize}
		\item adding $s+2$ new vertices $z_1^*,z_2^*,...,z_{s+2}^*$,
		\item for each $(A,B) \in \L$ and each 2-element subset $S$ of $[2,s+2]$, 
			\begin{itemize}
				\item adding a new vertex $u^*_{A,S}$ and all edges between $u^*_{A,S}$ and $U_A^{(1)}$, and
				\item adding new edges $u^*_{A,S}z^*_i$ for all $i \in [s+2]$ with $i \not\equiv i_A+j\epsilon_A$ (mod $s+2$), where $j\in S \cup \{1\}$.
			\end{itemize}
	\end{itemize}
Let $Z^*=\{z^*_{i}: i \in [s+2]\}$.
For each $(A,B) \in \L$, let $U_A^*=\{u^*_{A,S}: S \subseteq [2,s+2], \lvert S \rvert=2\}$.

\begin{claim}\label{claim_G^*_Kst}
$G^*$ does not contain a $K_{s,t+f_1(h(4\theta)+\theta)+(s+2)^2}$-subgraph.
\end{claim}	

\begin{proof}
Suppose to the contrary that some subgraph $Q$ of $G^*$ is isomorphic to $K_{s,t+f_1(h(4\theta)+\theta)+(s+2)^2}$.
Let $\{S,T\}$ be a bipartition of $V(Q)$ with $\lvert S \rvert=s$.

If there exist $(A,B) \in \L$ and $v \in S \cap (U_A^* \cup (V(A)-V(B))) \neq \emptyset$, then since $N_{G^*}(v) \subseteq Z_A \cup U_A^* \cup Z^*$, $v$ is adjacent in $G^*$ to at most $f_1(h(4\theta)+\theta)+{s+1 \choose 2}+(s+2) < f_1(h(4\theta)+\theta)+(s+2)^2$ vertices, so $v \not \in S$, a contradiction.

So $S \cap (U_A^* \cup (V(A)-V(B)))=\emptyset$ for every $(A,B) \in \L$.
Suppose that there exist $(A,B) \in \L$ and $u \in U_A^* \cap V(Q)$.
Then $u \in T$.
Since $u$ is adjacent to at most $s+2-3=s-1$ vertices in $Z^*$, $U_A^{(1)} \cap S \neq \emptyset$.
But $U_A^{(1)} \subseteq V(A)-V(B)$, a contradiction.

So there does not exist $(A,B) \in \L$ with $U^*_{A} \cap V(Q) \neq \emptyset$.
Hence if some vertex in $Z^*$ is in $V(Q)$, then it has no neighbor in $Q$, a contradiction.
So $V(Q) \cap Z^*=\emptyset$.
Therefore, $Q$ is a subgraph of $G'$, a contradiction.
\end{proof}

For each $i \in [\lvert \V \rvert]$, let $V^*_i = V_i \cup \bigcup_{(A,B) \in \L, U_A^{(1)} \subseteq V_i} U_A^*$.
Let $\V^*=(V_1^*,V_2^*,...,V_{\lvert \V \rvert}^*)$.
Note that $\V^*$ is a $(Y_1' \cup Z^*)$-layering of $G^*$. 

\begin{claim}\label{claim_G^*_ltw}
There exists a tree decomposition of $G^*$ with $\V^*$-width at most $w_0(\theta) \cdot (s+2)^2$.
\end{claim}	

\begin{proof}
Let $(T,\X)$ be the tree decomposition of $G'$ with $\V$-width at most $w_0(\theta)$ mentioned in \cref{NotControllingMinorClaim1}.
Add $Z^*$ into all bags of $(T,\X)$.
For each $(A,B) \in \L$, \cref{NotControllingMinorClaim1} implies that there exists a node $p_A$ of $T$ such that the bag at $p_A$ contains $U_A^{(1)}$. 
If $U_A^{(1)} \neq \emptyset$, then add $U_A^*$ into the bag at $p_A$; if $U_A^{(1)}=\emptyset$, then add a new node into $T$ adjacent to $p_A$ with bag $Z^* \cup U_A^*$.
We obtain a tree decomposition of $G^*$ with $\V^*$-width at most $w_0(\theta) \cdot (s+2)^2$.
\end{proof}

Let $L_\L$ be the following  list-assignment of $G^*$: 
	\begin{itemize}
		\item For every $v \in V(G')-Y_1'$, let $L_\L(v):=L'(v) \cap [s+2]-\{i\}$, where $i \in [s+2]$ is the number such that $v \in V_j$ and $j \equiv i$ (mod $s+2$).
		\item For every $v \in Y_1'$, let $L_\L(v):=L'(v) \cap [s+2]$.
		\item For every $i \in [s+2]$, let $L_\L(z^*_i)=\{i\}$.
		\item For every $(A,B) \in \L$ and $S \subseteq [2,s+2]$ with $\lvert S \rvert=2$, let $L_\L(u^*_{A,S})=\{i \in [s+2]: i \equiv i_A+j\epsilon_A$ (mod $s+2$) with $j\in S\}$.
	\end{itemize} 

Let $G''$ be the subgraph of $G^*$ induced by $\{v \in V(G^*): L_\L(v) \neq \emptyset\}$.
Note that for every $y \in Y_1'$, $y \in V(G'')$ if and only if $L'(y) \in [s+2]$.

\begin{claim}\label{claim_size_not_Y'}
For every $v \in V(G')-Y_1'$, $$\lvert L_\L(v) \rvert \geq \lvert L'(v) \rvert - r + \lvert N_G(v) \cap Y_1' \rvert - \lvert N_G(v) \cap Y_1' \cap V(G'') \rvert-1.$$
\end{claim}	

\begin{proof}
Since $L'$ satisfies (R5), for every $v \in V(G')-Y_1'$, 
\begin{align*}
	\lvert L_\L(v) \rvert & \geq \lvert L'(v) \rvert - \lvert L'(v) \cap [s+3,s+2+r] \rvert-1 \\
	& = \lvert L'(v) \rvert - (r-\lvert \{y \in N_{G}(v) \cap Y_1': L'(y) \subseteq [s+3,s+r+2]\}\rvert)-1 \\
	& = \lvert L'(v) \rvert - (r-\lvert N_{G}(v) \cap Y_1'-V(G'') \rvert)-1 \\
	& = \lvert L'(v) \rvert - r + \lvert N_G(v) \cap Y_1' \rvert - \lvert N_G(v) \cap Y_1' \cap V(G'') \rvert-1.\qedhere
\end{align*}
\end{proof}

For every $v \in V(G')-Y_1'$, since $L'$ is an $(s,r+2,Y_1')$-list-assignment, we know $\lvert L'(v) \rvert \geq r+3$, so $\lvert L_\L(v) \rvert \geq 2$ by \cref{claim_size_not_Y'}, and hence $v \in V(G'')$.
So $V(G^*)-V(G'') \subseteq Y_1'$.

Let $Y_1'' = (V(G'') \cap Y_1') \cup Z^*$.
Note that $Y_1'' = \{y \in V(G''): \lvert L_\L(y) \rvert=1\}$.

\begin{claim}\label{claim_size_weird_s}
For every $y \in V(G'') \cap N_{G''}^{<s}(Y_1'')$, there exists a subset $L_{\L,G''}(y)$ of $L_\L(y)$ with $\lvert L_{\L,G''}(y) \rvert = s+1-\lvert N_{G''}(y) \cap Y_1'' \rvert$ such that $L_{\L,G''}(y) \cap L_\L(u)=\emptyset$ for every $u \in N_{G''}(y) \cap Y_1''$.
\end{claim}	

\begin{proof}
Let $y \in V(G'') \cap N_{G''}^{<s}(Y_1'')$. 

First suppose $y \in N_{G''}(Z^*)$.
Then $y \in U_A^*$ for some $(A,B) \in \L$, so $\lvert L_\L(y) \rvert=2$ and $y=u^*_{A,S}$ for some $S \subseteq [2,s+2]$ with $\lvert S \rvert=2$.
Hence $N_{G''}(y) \cap Y_1'' = N_{G^*}(y) \cap Z^* = \{z^*_i: i \in [s+2], i \not \equiv i_A+j\epsilon_A, j \in S \cup \{1\}\}$.
So $\lvert N_{G''}(y) \cap Y_1'' \rvert = s+2-3=s-1$ and $L_\L(y) \cap L_\L(u)=\emptyset$ for every $u \in N_{G''}(y) \cap Y_1''$.
Hence $L_\L(y)$ is a desired subset $L_{\L,G''}(y)$ of $L_\L(y)$.

Now assume $y \not \in N_{G''}(Z^*)$.
Hence $y \in V(G')$ and $N_{G''}(y) \cap Y_1'' = N_{G''}(y) \cap Y_1' \cap V(G'') = N_{G}(y) \cap Y_1' \cap V(G'')$.
If $y \in N_{G}^{<s}(Y_1') \cup (V(G)-N_G[Y_1'])$, then since $L'$ is an $(s,r+2,Y_1')$-list-assignment of $G$, by \cref{claim_size_not_Y'},
	\begin{align*}
		\lvert L_\L(y) \rvert 
		& \geq \lvert L'(y) \rvert - r + \lvert N_G(y) \cap Y_1' \rvert - \lvert N_G(y) \cap Y_1' \cap V(G'') \rvert-1 \\
		& = (s+r+2 - \lvert N_G(y) \cap Y_1' \rvert) - r + \lvert N_G(y) \cap Y_1' \rvert - \lvert N_G(y) \cap Y_1' \cap V(G'') \rvert-1 \\
		& = s+1-\lvert N_G(y) \cap Y_1' \cap V(G'') \rvert \\
		& = s+1-\lvert N_{G''}(y) \cap Y_1'' \rvert,
	\end{align*} 
and $L_\L(y) \cap L_\L(u) \subseteq L'(y) \cap L'(u)=\emptyset$ for every $u \in N_G(y) \cap Y_1' \cap V(G'')=N_{G''}(y) \cap Y_1''$, so a desired subset $L_{\L,G''}(y)$ of $L_\L(y)$ exists.

So we may assume $y \in V(G')-( N_{G}^{<s}(Y_1') \cup (V(G)-N_G[Y_1']))$.
Hence $y \in N_{G}^{\geq s}(Y_1')$.
In particular, $y \not \in Y_1'$, so $y \not \in N_G^{\geq s}(Y_1 \cup Z)$ since $L'$ is a $(Z,\ell)$-growth of $L$.
For each $i \in [0,s+2]$, let $L^{(i)}$ and $Y_1^{(i)}$ be the list-assignment $L^{(i)}$ and the subset $Y_1^{(i)}$ of $V(G)$ mentioned in the definition of a $(Z,\ell)$-growth.
Since $y \in N_{G}^{\geq s}(Y_1')-N_G^{\geq s}(Y_1 \cup Z) = N_G^{\geq s}(Y_1^{(s+2)})-N_G^{\geq s}(Y_1^{(0)})$, there exists $j \in [s+2]$ such that $y \in N_G^{\geq s}(Y_1^{(j)})-N_G^{\geq s}(Y_1^{(j-1)})$.
So $\lvert L'(y) \rvert = \lvert L^{(j)}(y) \rvert  = \lvert L^{(j-1)}(y) \rvert \geq s+r+2-\lvert N_G(y) \cap Y_1^{(j-1)} \rvert$, and for every $u \in N_G(y) \cap Y_1^{(j-1)}$, $L_\L(y) \cap L_\L(u) \subseteq L'(y) \cap L'(u) \subseteq L^{(j-1)}(y) \cap L^{(j-1)}(u)=\emptyset$.
Hence by \cref{claim_size_not_Y'}, since $Y_1' \supseteq Y_1^{(j-1)}$,
	\begin{align*}
		\lvert L_\L(y) \rvert 
		& \geq \lvert L'(y) \rvert - r + \lvert N_G(y) \cap Y_1' \rvert - \lvert N_G(y) \cap Y_1' \cap V(G'') \rvert-1 \\
		& \geq (s+r+2-\lvert N_G(y) \cap Y_1^{(j-1)} \rvert) - r + \lvert N_G(y) \cap Y_1' \rvert - \lvert N_G(y) \cap Y_1' \cap V(G'') \rvert-1 \\
		& = s+1-\lvert N_G(y) \cap Y_1' \cap V(G'') \rvert + (\lvert N_G(y) \cap Y_1' \rvert-\lvert N_G(y) \cap Y_1^{(j-1)} \rvert) \\
		& = s+1-\lvert N_{G''}(y) \cap Y_1'' \rvert + (\lvert N_G(y) \cap Y_1' - Y_1^{(j-1)} \rvert),
	\end{align*} 
So there exists a subset $L_{\L,G''}(y)$ of $L_\L(y)$ with size $s+1-\lvert N_{G''}(y) \cap Y_1'' \rvert$ such that $L_{\L,G''}(y) \cap L_\L(u)=\emptyset$ for every $u \in N_G(y) \cap Y_1'-Y_1^{(j-1)}$.
Hence $L_{\L,G''}(y) \cap L_\L(u)=\emptyset$ for every $u \in N_G(y) \cap Y_1' \cap V(G'')=N_{G''}(y) \cap Y_1''$.
\end{proof}

\begin{claim}\label{claim_better_list}
There exists an $(s,1,Y_1'')$-list-assignment $L_{\L,G''}$ of $G''$ such that $L_{\L,G''}(v) \subseteq L_\L(v)$ for every $v \in V(G'')$.
\end{claim}	

\begin{proof}
For every $v \in Y_1''$, let $L_{\L,G''}(v)=L_\L(v)$.
For every $v \in  V(G'') \cap N_{G''}^{<s}(Y_1'')$, let $L_{\L,G''}(v)$ be the set $L_{\L,G''}(v)$ mentioned in \cref{claim_size_weird_s}.
So $L_{\L,G''}$ satisfies (L3) by \cref{claim_size_weird_s}.

For every $v \in V(G'')-N_{G''}[Y_1'']$ with $v \in V(G'')-N_{G}[Y_1']$, we know either $v \in V(G')-Y_1'$, or $s=1$ and $v \in \bigcup_{(A,B) \in \L}U_A^*$; for the former, by \cref{claim_size_not_Y'},
	\begin{align*}
		\lvert L_\L(v) \rvert 
		& \geq \lvert L'(v) \rvert - r + \lvert N_G(v) \cap Y_1' \rvert - \lvert N_G(v) \cap Y_1' \cap V(G'') \rvert-1 \\
		&  = \lvert L'(v) \rvert - r -1 \\
		& \geq (s+r+2)-r-1 \\ 
		& = s+1;
	\end{align*} 
for the latter, $\lvert L_\L(v) \rvert=2=s+1$.
So there exists $L_{\L,G''}(v) \subseteq L_\L(v)$ with $\lvert L_{\L,G''}(v) \rvert = s+1$ in each case.

Let $v \in V(G'')-N_{G''}[Y_1'']$ with $v \not \in V(G'')-N_{G}[Y_1']$.
So $v \in V(G'') \cap N_G[Y_1'] - N_{G''}[Y_1'']$.
Hence $v \in V(G')-Y_1'$, and for every $y \in N_G(v) \cap Y_1'$, $L'(y) \in [s+3,s+r+2]$ and hence $y \in Y_1$ (since $L'$ is a $(Z,\ell)$-growth of $L$).
So $N_G(v) \cap Y_1' = N_G(v) \cap Y_1$.
This implies $L'(v)=L(v)$ since $L'$ is a $(Z,\ell)$-growth of $L$.

Since $v \in N_G[Y_1']-Y_1'$, we know $v \in N_G^{<s}(Y_1') \cup N_G^{\geq s}(Y_1')$.
Since $N_G(v) \cap Y_1' = N_G(v) \cap Y_1$, $v \in N_G^{<s}(Y_1) \cup N_G^{\geq s}(Y_1)$.
Since $v \not \in Y_1'$, $v \not \in N_G^{\geq s}(Y_1)$.
So $v \in N_G^{<s}(Y_1)$.
Hence $\lvert L(v) \rvert = s+r+2-\lvert N_G(v) \cap Y_1 \rvert$.
By \cref{claim_size_not_Y'},
	\begin{align*}
	\lvert L_\L(v) \rvert 
	& \geq \lvert L'(v) \rvert - r + \lvert N_G(v) \cap Y_1' \rvert - \lvert N_G(v) \cap Y_1' \cap V(G'') \rvert-1 \\
	& = \lvert L(v) \rvert - r + \lvert N_G(v) \cap Y_1 \rvert - \lvert N_G(v) \cap Y_1' \cap V(G'') \rvert-1 \\
	& = (s+r+2-\lvert N_G(v) \cap Y_1 \rvert) - r + \lvert N_G(v) \cap Y_1 \rvert - \lvert N_G(v) \cap Y_1' \cap V(G'') \rvert-1 \\
	& = s+1-\lvert N_G(v) \cap Y_1' \cap V(G'') \rvert.
	\end{align*}
Since $v \not \in N_{G''}[Y_1'']$, $\lvert N_G(v) \cap Y_1' \cap V(G'') \rvert=0$.
So $\lvert \L_\L(v) \rvert \geq s+1$.
Therefore, there exists $L_{\L,G''}(v) \subseteq L_\L(v)$ with $\lvert L_{\L,G''}(v) \rvert = s+1$.
So $L_{\L,G''}$ satisfies (L4).

Moreover, for every $x \in V(G'')-Y_1''$ in which $L_{\L,G''}(x)$ has not been defined, since $\lvert L_\L(x) \rvert \geq 2$ for every $x \in V(G'')-Y_1''$, there exists $L_{\L,G''}(x) \subseteq L_\L(x)$ with $\lvert L_{\L,G''}(x) \rvert = 2$.
This implies that $L_{\L,G''}$ is an $(s,1,Y_1'')$-list-assignment of $G''$ and proves the claim.
\end{proof}

Let $L_{\L,G''}$ be an $(s,1,Y_1'')$-list-assignment of $G''$ stated in \cref{claim_better_list}.

Recall that $V(G^*)-V(G'') \subseteq Y_1'$ and $\V^*$ is a $(Y_1' \cup Z^*)$-layering of $G^*$.
So $\V^*$ is a $Y_1''$-layering of $G''$.
Note that $L_\L|_{V(G'')}$ is $(s,\V^*)$-compatible.
So $L_{\L,G''}$ is $(s,\V^*)$-compatible.
Since $L_{\L,G''}$ is an $(s,1,Y_1'')$-list-assignment of $G''$, $(Y_1'', L_{\L,G''})$ is a $\V^*$-standard pair.

\begin{claim}\label{claim_coloring}
There exists an $(L' \cup L_\L)$-coloring $c_\L$ of $G^*$ with clustering $\eta_1(\eta+\theta)$ such that:
	\begin{itemize}
		\item $c_\L|_{V(G'')}$ is an $L_{\L,G''}$-coloring of $G''$,  
		\item for every $x \in [s+3,s+2+r]$, $\{v \in V(G^*): c_\L(v)=x\}=\{v \in Y_1': L'(v)=\{x\}\}$ is a stable set, and
		\item for every $(A,B) \in \L$, there exists at most one element $k_A$ in $[s+r+2]$ such that there exist two monochromatic components with respect to $c_\L$ each having color $k_A$ and intersecting $U_A^{(1)}$.
	\end{itemize}
\end{claim}	

\begin{proof}
Since $G''$ is an induced subgraph of $G^*$, by \cref{claim_G^*_Kst,claim_G^*_ltw}, $G''$ does not contain a $K_{s,t+f_1(h(4\theta)+\theta)+(s+2)^2}$ subgraph, and there exists a tree-decomposition of $G''$ with $\V^*$-width at most $(s+2)^2 \cdot w_0(\theta)$.
By \cref{Z growth}, $\lvert Y_1'' \rvert \leq \lvert Y_1' \rvert+\lvert Z^* \rvert \leq h(\lvert Y_1 \rvert + \lvert Z \rvert)+s+2 \leq h(\eta+\xi_0)+s+2$.
Since $L_{\L,G''}$ is $(s,\V^*)$-compatible and $(Y_1'', L_{\L,G''})$ is a $\V^*$-standard pair, by \cref{bounded layered tw}, there exists an $L_{\L,G''}$-coloring $c_\L$ of $G''$ with clustering $\eta_1(\eta+\theta)$.
By further coloring each vertex $y$ in $Y_1'-V(G'')$ with the unique element in $L'(y)$, we extend the coloring $c_\L$ to be an $(L' \cup L_\L)$-coloring of $G^*$ with clustering $\eta_1(\eta+\theta)$, and $\{v \in V(G^*): c_\L(v)=x\}=\{v \in Y_1': L'(v)=\{x\}\}$ is a stable set for every $x \in [s+3,s+2+r]$.

Let $(A,B) \in \L$.
Let $M_1$ and $M_2$ be two distinct monochromatic components with respect to $c_\L$ having the same color and intersecting $U_A^{(1)}$.
Let $k$ be the color of $M_1$ and $M_2$.
Since $M_1$ intersects $U_A^{(1)} \subseteq V_{i_A+\epsilon_A}$, we know $k \in [s+2]$, and by the definition of $L_\L$, $k \not \equiv i_A+\epsilon_A$ (mod $s+2$).
That is, $k \equiv i_A+j\epsilon_A$ (mod $s+2$) for some $j \in [2,s+2]$.
Since every vertex in $U_A^*$ is adjacent to every vertex in $U_A^{(1)}$, no vertex in $U_A^*$ is colored with $k$ by $c_\L$, for otherwise, $M_1$ and $M_2$ are not distinct.
Hence, if there are two possibilities for $k$, then there exists $S \subseteq [2,s+2]$ with $\lvert S \rvert=2$ such that $u^*_{A,S}$ cannot be colored by $c_\L$ by the definition of $L_\L(u^*_{A,S})$, a contradiction.
\end{proof}

Let $c_\L$ be the $(L' \cup L_\L)$-coloring of $G^*$ mentioned in \cref{claim_coloring}.

\begin{claim}\label{NotControllingMinorClaim2}
For each $(A,B) \in \L$, there exist $Y_A \subseteq V(A)$ with $Z_A \subseteq Y_A$ and with $\lvert Y_A \rvert \leq \eta^*(\theta)$ and an $(s,Y_A, \ell,r)$-list-assignment $L_A$ of $G[V(A)]$ such that:
	\begin{itemize}
		\item $L_A(v)=\{c_\L(v)\}$ for every $v \in Z_A$, 
		\item $L_A(v) \subseteq L'(v)$ for every $v \in V(A)$, and 
		\item for every $L_A$-coloring $c_A$ of $G[V(A)]$, every monochromatic component $M$ of $G[V(A)]$ with respect to $c_A$ intersecting $(Y_1' \cap V(A)) \cup V(A \cap B)=U_A^{(0)}$ either:
			\begin{itemize}
				\item is contained in $G[Z_A]$, or 
				\item contains a vertex in $Z_A-V(B)$, is contained in $G[Y_A]$, and some monochromatic component of $G^*$ with respect to $c_\L$ contains $M[V(M) \cap Z_A]$.
			\end{itemize}
	\end{itemize}
\end{claim}

\begin{proof}
For every separation $(A,B)$, define the following:
	\begin{itemize}
		\item Let $Y_A':=Z_A$.
		\item For every $v \in Y_A'$, let $L_A'(v):=\{c_\L(v)\}$.
		\item For every $v \in V(A) - Y_A'$ with $1 \leq \lvert N_{G[V(A)]}(v) \cap Y_A' \rvert \leq s-1$, let $L_A'(v)$ be a subset of $L'(v)$ with size $s+r+2-\lvert N_{G[V(A)]}(v) \cap Y_A' \rvert$ such that $L'_A(v) \cap L'_A(y)=\emptyset$ for every $y \in N_{G[V(A)]}(v) \cap Y_A'$, and $\lvert \{y \in N_G(v) \cap Y_A': L_A'(y) \subseteq [s+3,s+2+r]\} \rvert= r-\lvert L_A'(v) \cap [s+3,s+2+r] \rvert$.
		(Such $L_A'(v)$ exists since $v \in V(A)-Y_A'$ implies $v \in V(A)-V(B)$, and $L'$ satisfies (R5).)
	\item For every $v \in V(A)-Y_A'$ with $\lvert N_{G[V(A)]}(v) \cap Y_A' \rvert \geq s$ and $1 \leq \lvert N_{G[V(A)]}(v) \cap U_A^{(0)} \rvert \leq s-1$, let $L_A'(v)$ be a subset of $L'(v)$ with size $s+r+2-\lvert N_{G[V(A)]}(v) \cap U_A^{(0)} \rvert$ such that $L_A'(v) \cap L_A'(y)=\emptyset$ for every $y \in N_{G[V(A)]}(v) \cap U_A^{(0)}$, and $\lvert \{y \in N_G(v) \cap Y_A': L_A'(y) \subseteq [s+3,s+2+r]\} \rvert= r-\lvert L_A'(v) \cap [s+3,s+2+r] \rvert$.
		(Such $L_A'(v)$ exists since $v \in V(A)-Y_A'$ implies $v \in V(A)-V(B)$, and $Y_1' \cap V(A) \subseteq U_A^{(0)}$ and $L'$ satisfies (R5).)
		\item For every other vertex $v$ in $V(A)$, let $L_A'(v):=L'(v)$.
	\end{itemize}
Since $V(A \cap B) \subseteq Y_A'$, and $L'$ is an $(s,r+2,Y_1')$-list-assignment of $G$, $L_A'$ is an $(s,r+2,Y_A')$-list-assignment of $G[V(A)]$.
Since $\{v \in V(G^*): c_\L(v)=x\}=\{v \in Y_1': L'(v)=\{x\}\}$ is a stable set for every $x \in [s+3,s+2+r]$, and $L'$ satisfies (R3) and (R4) (with $Y_1$ replaced by $Y_1'$), we know $L_A'$ satisfies (R4) (with $Y_1$ replaced by $Y_A'$), and $\{v \in N_{G[V(A)]}(y)-Y_A': x \in L'_A(v)\}=\emptyset$ for every $y \in Y_A'$ and $x \in [s+3,s+2+r]$ with $x \in L_A'(y)$.
Since $Y_1' \cap V(A) \subseteq U_A^{(0)}$ and $N_{G[V(A)]}^{\geq s}(U_A^{(0)}) \subseteq U_A^{(1)} \subseteq Y_A'$ and $L'$ satisfies (R5) (with $Y_1$ replaced by $Y_1'$), $L_A'$ satisfies (R5) (with $Y_1$ replaced by $Y_A'$) by the definition of $L_A'$.
In addition, $\lvert Y_A' \rvert = \lvert Z_A \rvert \leq f_{1}(h(4\theta)+\theta)$.

For each $(A,B) \in \L$, if the element $k_A$ in $[s+r+2]$ mentioned in \cref{claim_coloring} exists, then let $k_A$ be this element; otherwise, let $k_A=0$.
Note that $k_A \not \in [s+3,s+r+2]$ since if $k_A$ exists in \cref{claim_coloring}, then the corresponding monochromatic component intersects $U_A^{(1)}$.

For each $(A,B) \in \L$, we further define the following:
	\begin{itemize}
		\item Let $W_A:=\{v \in V(A)-Y_A': \lvert N_{G[V(A)]}(v) \cap Y_A' \rvert \geq s\}$.
		\item Let $L'_{A,W}$ be a $(W_A,\{k_A\} \cup [s+3,s+2+r])$-progress of $L_A'$ (in $G[A]$).
		\item Let $W_A'=\{v \in V(A)-(Y_A' \cup W_A): \lvert N_{G[V(A)]}(v) \cap (Y_A' \cup W_A) \rvert \geq s\}$.
		\item Let $L''_A$ be a $(W_A',\{\ell\} \cup [s+3,s+2+r])$-progress of $L_{A,W}'$ (in $G[A]$).
		\item Let $Y''_A:=Y_A' \cup W_A \cup W_A'$. 
	\end{itemize}
Since $L_A'$ is an $(s,r+2,Y_A')$-list-assignment of $G[V(A)]$, we have $L''_A$ satisfies (R1)--(R3) (with $Y_1$ replaced by $Y''_A$) by \cref{progress basic}.
By the definition of $L''_A$, for every $x \in [s+3,s+2+r]$, we have $\{y \in Y''_A: x \in L''_A(y)\} = \{y \in Y_A': x \in L_A'(y)\}$, so $L''_A$ satisfies (R4) and (R5) (with $Y_1$ replaced by $Y''_A$) by \cref{progress basic}.
Hence $L''_A$ is an $(s,Y''_A,\ell,r)$-list-assignment of $G[V(A)]$.
Note that by \cref{BoundedGrowth}, $$\lvert Y''_A \rvert \leq \lvert Y_A' \rvert + \lvert W_A \rvert+\lvert W_A' \rvert \leq \lvert Y_A' \rvert + f(\lvert Y_A' \rvert) + f(\lvert Y_A' \rvert + f(\lvert Y_A' \rvert)) \leq f_3(h(4\theta)+\theta).$$

For each $(A,B) \in \L$, define $L_A$ to be a $(Y_A'',\ell)$-growth of $L_A''$ (in $G[V(A)]$), and let $Y_A = \{v \in V(A): \lvert L_A(v) \rvert=1\}$.
By \cref{Z growth}, $L_A$ is an $(s,Y_A,\ell,r)$-list-assignment of $G[V(A)]$.
Furthermore, it is clear that $L_A(v) \subseteq L_A''(v) \subseteq L_A'(v) \subseteq L'(v)$ for every $v \in V(A)$, so $L_A(v)=\{c_\L(v)\}$ for every $v \in Y_A'$, and every $L_A$-coloring is an $L_A'$-coloring and an $L_A''$-coloring.
By \cref{Z growth}, 
$$\lvert Y_A \rvert \leq h(\lvert Y_A'' \rvert) \leq h(f_3(h(4\theta)+\theta)) = \eta^*(\theta).$$
In addition, for every $(A,B) \in \L$, we have $Y_1 \cap V(A) \subseteq Y_A' \subseteq Y_A'' \subseteq Y_A \subseteq V(A)$.

Now we prove the last statement of this claim.
Assume that $(A,B)$ is a fixed element of $\L$ and $M$ is a fixed monochromatic component with respect to a fixed $L_A$-coloring $c_A$ of $G[V(A)]$ intersecting $(Y_1' \cap V(A)) \cup V(A \cap B)=U_A^{(0)}$.
We may assume that $M$ is not contained in $G[Z_A]$ for otherwise we are done.

Since $L_A(v) \subseteq L'(v)$ for every $v \in V(A)$, $c_A$ is an $L'|_{G[V(A)]}$-coloring of $G[V(A)]$.
We can extend $c_A$ to be an $L'$-coloring $c'$ of $G$ by coloring each vertex $v$ in $V(G)-V(A)$ with an arbitrary color in $L'(v)$.
So $M$ is a connected subgraph of $M'[V(M') \cap G[A]]$ for some monochromatic component $M'$ with respect to $c'$ intersecting $(Y_1' \cap V(A)) \cup V(A \cap B)$.
Since $G[Y_1' \cap V(A)] \subseteq G[Z_A]$, if $V(M) \cap (Y_1 \cup Z) \neq \emptyset$, then $M'$ intersects $Y_1 \cup Z$, so $M'$ is contained in $G[Y_1']$ by \cref{Z growth} (since $L'$ is a $(Z,\ell)$-growth of $L$), and hence $V(M) \subseteq Y_1' \cap V(A) \subseteq Z_A$, a contradiction.
So $V(M) \cap (Y_1 \cup Z)=\emptyset$.
Hence $V(M) \cap ((Y_1' \cap V(A)-(Y_1 \cup Z)) \cup (V(A \cap B)-Y_1')) \neq \emptyset$.

Note that $(Y_1' \cap V(A)-(Y_1 \cup Z)) \cup (V(A \cap B)-Y_1') \subseteq U_A^{(0)} \subseteq Z_A$.
So $M$ intersects $U_A^{(0)} \subseteq Z_A$ and $V(M) \not \subseteq Z_A$.
Hence there exist $v \in V(M)-Z_A$ and $y \in V(M) \cap Z_A$ such that $y \in N_{G[V(A)]}(v) \cap Z_A$ and $L''_A(v) \cap L''_A(y) \neq \emptyset$.
Since $Z_A \subseteq Y''_A$ and $L''_A$ is an $(s,Y''_A,\ell,r)$-list-assignment, (R3) and (R4) imply that the vertices of $M$ are colored with some color in $[s+2]$.
So $V(M) \cap Z_A \subseteq V(G'')$.

Since $M$ intersects $U_A^{(0)}$, the definition of $L'_A(v)$ implies that $M$ intersects $U_A^{(1)}$.
In particular, $M$ contains a vertex in $Z_A-V(B)$.
Moreover, since $L_A$ is a $(Y_A'',\ell)$-growth of $L_A''$ and $M$ intersects $U_A^{(0)} \subseteq Y_A''$, $M$ is contained in $G[Y_A''] \subseteq G[Y_A]$ by \cref{Z growth}.

Hence, to prove this claim, it suffices to show that some monochromatic component of $G^*$ respect to $c_\L$ contains $M[V(M) \cap Z_A]$. 

Since $Y_A'=Z_A$ and $L'_A(v) \cap L'_A(y) \supseteq L''_A(v) \cap L''_A(y) \neq \emptyset$, the definition of $L'_A$ implies that $v \in W_A$.
By the definition of $L'_{A,W}$, $M$ does not have color $k_A$.
Note that every component of $M[V(M) \cap Z_A]$ is contained in some monochromatic component of $G^*$ with respect to $c_\L$ and either intersects $U_A^{(0)}$ or intersects $U_A^{(1)}$.
By the definition of $L'_A$, every component of $M[V(M) \cap Z_A]$ intersects $U_A^{(1)}$.
Since $M$ does not have color $k_A$, by \cref{claim_coloring}, some monochromatic component of $G^*$ with respect to $c_\L$ contains $M[V(M) \cap Z_A]$.
This proves the claim.
\end{proof}

For every $(A,B) \in \L$, let $L_A$ and $Y_A$ be the list-assignment and set mentioned in \cref{NotControllingMinorClaim2}, respectively, so $\lvert Y_A \rvert \leq \eta^*(\theta)$.
For every $(A,B) \in \L$, since $Y_1 \cap V(A) \subseteq Y_A \subseteq V(A)$, there exists an $(\eta,g)$-bounded $L_A$-coloring $c_A$ of $G[V(A)]$ such that for every $x \in [s+3,s+2+r]$, the set of vertices colored $x$ is a stable set by our assumption.
Since $\Se$ is a segregation of $G-Z$ and $\L$ is a location, for every $v \in V(G)-V(G')$, there uniquely exists $(A_v,B_v) \in \L$ such that $v \in V(A_v)-V(B_v)$.
Let $c^*$ be the function, where $c^*(v):=c_\L(v)$ for every $v \in V(G')$, and
$c^*(v):=c_{A_v}(v)$ for every $v \in V(G)-V(G')$. 
Clearly, $c^*$ is an $L'$-coloring (and hence an $L$-coloring) of $G$.
Note that for every $v \in V(A) \cap V(G')$ for some $(A,B) \in \L$, $c^*(v)=c_\L(v)=c_{A_v}(v)$ by the property of $L_A$.

Suppose that there exists $x^* \in [s+3,s+2+r]$ such that the set of vertices colored $x^*$ is a not stable set.
Then there exists an edge $e$ of $G$ whose both ends are colored $x^*$.
Since for every $(A,B) \in \L$, the set of vertices colored $x^*$ under $c_A$ is a stable set, $e$ belongs to $\bigcap_{(A,B) \in \L}B \subseteq G'$.
But the set of vertices colored  $x^*$ under $c_\L$ is a stable set by \cref{claim_coloring}, a contradiction.

So $c^*$ is not $(\eta,g)$-bounded by our assumption.
By \cref{Z growth}, the union $U^*$ of the monochromatic components with respect to $c^*$ intersecting $Y_1 \cup Z$ is contained in $G[Y_1']$.
Note that $\lvert Y_1 \rvert >3\theta \geq 1$.
So $U^*$ contains at most $\lvert Y_1' \rvert \leq h(\lvert Y_1 \rvert+\xi_0) \leq g^*(\lvert Y_1 \rvert) \leq g(\lvert Y_1 \rvert) \leq \lvert Y_1 \rvert^2g(\lvert Y_1 \rvert)$ by \cref{Z growth}.
Note that $\lvert Y_1 \rvert \leq \eta$.
Hence, there exists a monochromatic component $M$ with respect to $c^*$ disjoint from $Y_1 \cup Z$ containing more than $\eta^2g(\eta)$ vertices.

\begin{claim} \label{claim_frame_link}
There exists a monochromatic component of $G^*$ with respect to $c_\L$ containing $M[V(M) \cap V(G')]$.
\end{claim}

\begin{proof}
Let $P$ be a path in $M$ between two vertices $x,y$ in $V(M) \cap V(G')$ internally disjoint from $V(M) \cap V(G')$ having at least one internal vertex.
To prove this claim, it suffices to show that there exists a monochromatic component of $G^*$ with respect to $c_\L$ containing both $x$ and $y$.

Since $G' \supseteq G[\bigcap_{(A,B) \in \L} V(B)]$, there exists $(A,B) \in \L$ such that $P \subseteq A$.
Let $M_x$ and $M_y$ be the components of $M[V(M) \cap V(A) \cap V(G')]$ contaning $x$ and $y$, respectively.
By the existence of $P$, there exists a component $M_P$ of $M[V(M) \cap V(A)]$ containing $M_x$ and $M_y$, and $M_P$ is not contained in $G[Z_A]$.
Hence $V(M) \not \subseteq V(A)$, for otherwise $M=M_P$ is a monochromatic component with respect to $c_A$ and hence contains at most $\eta^2g(\eta)$ vertices, a contradiction.
So $M_P$ is a monochromatic component with respect to $c_A$ intersecting $(Y_1' \cap V(A)) \cup V(A \cap B)$ but not contained in $G[Z_A]$.
By \cref{NotControllingMinorClaim2}, some monochromatic component of $G^*$ with respect to $c_\L$ contains $M_P[V(M_P) \cap Z_A]$ and hence contains both $x$ and $y$.
\end{proof}

For every $(A,B) \in \L$, let $S_A$ be the set of the components of $M[V(M) \cap V(A)]$.
Let $\L^* = \{(A,B) \in \L: V(M)-V(B) \neq \emptyset\}$.

\begin{claim} \label{claim_special_sep}
For every $(A,B) \in \L^*$, $V(M) \cap Z_A -V(B) \neq \emptyset$ and $\sum_{C \in S_A}\lvert V(C) \rvert \leq \eta^*(\theta)$. 
\end{claim}

\begin{proof}
Let $(A,B) \in \L^*$.
Since $c_A$ is $(\eta,g)$-bounded, $V(M) \not \subseteq V(A)$, so each member of $S_A$ is a monochromatic component of $G[A]$ with respect to $c_A$ intersecting $(Y_1' \cap V(A)) \cup V(A \cap B)$.
So by \cref{NotControllingMinorClaim2}, for every member $C$ of $S_A$, either $C$ is contained in $G[Z_A]$, or $C$ contains a vertex in $Z_A-V(B)$ and is contained in $G[Y_A]$.
Hence if some member of $S_A$ satisfies the latter, then $V(M) \cap Z_A-V(B) \neq \emptyset$; if every member of $S_A$ satisfies the former, then $V(M) \cap V(A) \subseteq Z_A$, so $V(M) \cap Z_A-V(B) = V(M)-V(B) \neq \emptyset$ since $(A,B) \in \L^*$.
Therefore, $V(M) \cap Z_A-V(B) \neq \emptyset$.
Moreover, since $Z_A \subseteq Y_A$, $\sum_{C \in S_A}\lvert V(C) \rvert \leq \lvert Y_A \rvert \leq \eta^*(\theta)$.
\end{proof}

By \cref{claim_coloring,claim_frame_link}, $\lvert V(M) \cap \bigcap_{(A,B) \in \L}V(B) \rvert \leq \lvert V(M \cap G') \rvert \leq \eta_1(\eta+\theta)$.
By \cref{claim_special_sep}, $\lvert \L^* \rvert \leq \lvert V(M \cap G') \rvert \leq \eta_1(\eta+\theta)$, and 
	\begin{align*}
		\lvert V(M) \rvert & = \lvert V(M) \cap \bigcap_{(A,B) \in \L}V(B) \rvert + \sum_{(A,B) \in \L^*, C \in S_A}\lvert V(C)-V(B) \rvert \\
		& \leq \eta_1(\eta+\theta) + \lvert \L^* \rvert \cdot \eta^*(\theta) \\ 
		& \leq \eta_1(\eta+\theta) + \eta_1(\eta+\theta) \cdot \eta^*(\theta) \\ 
		& = \eta_1(\eta+\theta) \cdot (1+ \eta^*(\theta)) \\ 
		& \leq \eta_1(2\eta) \cdot (1+ \eta^*(\eta)) \\ 
		& \leq g^*(\eta) \leq \eta^2g(\eta),
	\end{align*}
a contradiction.
This completes the proof.
\end{proof}

We now prove our main theorems. Before proving \cref{OddMinorFree}, we note that it implies our main results from \cref{Intro}.  In particular, the first part of \cref{OddMinorFree} with $Y_1=\emptyset$ implies \cref{minorbasic} since the list-assignment $(L(v): v \in V(G))$ with $L(v)=[s+2]$ for every $v \in V(G)$ is an $(s,\emptyset,0,0)$-list-assignment. 
The second part of \cref{OddMinorFree} with $Y_1=\emptyset$ and $\ell=1$ implies \cref{oddminorbasic} since the list-assignment $(L(v): v \in V(G))$ with $L(v)=[2s+1]$ for every $v \in V(G)$ is an $(s,\emptyset,1,s-1)$-list-assignment. 

\begin{theorem} 
\label{OddMinorFree}
For all $s,t\in\mathbb{N}$ and for every graph $H$, there exists $\eta\in\mathbb{N}$ and a nondecreasing function $g$ such that the following hold:
	\begin{enumerate}
		\item If $G$ is a graph with no $K_{s,t}$ subgraph and no $H$-minor, $Y_1 \subseteq V(G)$ with $\lvert Y_1 \rvert \leq \eta$, and $L$ is an $(s,Y_1,0,0)$-list-assignment of $G$, then there exists an $(\eta,g)$-bounded $L$-coloring.
		\item If $G$ is a graph with no $K_{s,t}$ subgraph and no odd $H$-minor, $Y_1 \subseteq V(G)$ with $\lvert Y_1 \rvert \leq \eta$, $\ell \in [s+2]$ and $L$ is an $(s,Y_1,\ell,s-1)$-list-assignment of $G$, then there exists an $(\eta,g)$-bounded $L$-coloring such that for every $x \in [s+3,2s+1]$, the set of vertices colored $x$ is a stable set in $G$.
	\end{enumerate}
\end{theorem}

\begin{proof}
Define the following:
	\begin{itemize}
		\item Let $f$ be the function $f_{s,t}$  in \cref{BoundedGrowth}.
		\item Let $f_0:\mathbb{N}_0\to \mathbb{N}_0$ be the identity function, and for every $i\in\mathbb{N}$, let $f_i:\mathbb{N}_0\to \mathbb{N}_0$ be the function defined by $f_i(x):=f_{i-1}(x)+f(f_{i-1}(x))$.
		\item Let $C_0$ be the integer $c$ in \cref{structure odd minor}.
		\item Let $t':=\lceil (C_0+9) \lvert V(H) \rvert\sqrt{\log {12\lvert V(H) \rvert}} \rceil$.
		\item Let $\theta_0$ be the number $\theta^*$ and let $g_0,\eta_0$ be the functions $g^*,\eta^*$, respectively, in \cref{NotControllingMinor}  taking $s=s$, $t=t$ and $t'=t'$.
		\item Let $\xi:=8\lvert V(H) \rvert$.
		\item Let $h:\mathbb{N}_0\to \mathbb{N}_0$ be the function mentioned in \cref{Z growth} by taking $s=s$ and $t=t$.
		\item Let $\theta' :=\theta_0+t'+1$.
		\item Let $\eta_1,g_1$ be the number $\eta$ and the function $g$, respectively, mentioned in \cref{enlarge Y} by taking $s=s$, $t=t$ and $k=9\theta'$.
		\item Let $\eta_3:=h(\theta'+h(4\theta')+1 + f(h(4\theta')+1))$.
		\item Define $\eta:= \eta_0(\theta') + \eta_1 + f(\eta_1) + h(4\theta')+1 + f(h(4\theta')+1) + \eta_3$.
		\item Let $\eta_2:=h(\eta+\xi) + f(h(\eta+\xi))$. 
		\item Define $g:\mathbb{N}\to\mathbb{N}$ be the function defined by $g(x):=x+g_0(x)+g_1(x)+\eta_1 + \eta_2 \cdot \eta_3$ for every $x\in\mathbb{N}$. 
	\end{itemize}

Suppose to the contrary that there exists a graph $G$ with no $K_{s,t}$ subgraph, a subset $Y_1$ of $V(G)$ with $\lvert Y_1 \rvert \leq \eta$, a number $\ell \in [s+2]$, and list-assignment $L$ of $G$ such that:
	\begin{itemize}
		\item If $G$ has no $H$-minor, then $L$ is an $(s,Y_1,0,0)$-list-assignment such that there exists no $(\eta,g)$-bounded $L$-coloring of $G$.
		\item If $G$ has no odd $H$-minor, then $L$ is an $(s,Y_1,\ell,s-1)$-list-assignment of $G$ such that there exists no $(\eta,g)$-bounded $L$-coloring $c$ of $G$ such that  $\{v \in V(G): c(v)=x\}$ is a stable set in $G$ for every $x \in [s+3,2s+1]$.
	\end{itemize} 
We further choose $G$ and $Y_1$ so that $\lvert V(G) \rvert$ is as small as possible, and subject to this, $\lvert Y_1 \rvert$ is as large as possible.
Let $r:=0$ and $\ell':=0$ when $G$ has no $H$-minor; let $r:=s-1$ and $\ell':=\ell$ when $G$ has no odd $H$-minor.

\begin{claim}\label{OddMinorFreeClaim1}
$Y_1 \neq \emptyset$ and $N_G(Y_1) \neq \emptyset$.
\end{claim}

\begin{proof}
If $Y_1=\emptyset$, then let $v$ be a vertex of $G$ and define $L'$ to be a $(\{v\}, \{\ell\} \cup [s+3,2s+1])$-progress of $L$. 
Let $Y_1'=\{v\}$.
By \cref{progress basic}, $L'$ is an $(s,Y_1',0,0)$-list-assignment when $G$ has no $H$-minor; $L'$ is an $(s,Y_1',\ell,s-1)$-list-assignment when $G$ has no odd $H$-minor.
In addition, $\lvert Y_1' \rvert \leq \eta$.
So the maximality of $Y_1$ implies that there exists an $(\eta,g)$-bounded $L'$-coloring $c'$ of $G$ such that if $G$ has no odd $H$-minor, then $\{v \in V(G): c'(v)=x\}$ is a stable set in $G$ for every $x \in [s+3,2s+1]$.
But $c'$ is an $(\eta,g)$-bounded $L$-coloring $c$ of $G$ such that if $G$ has no odd $H$-minor, then $\{v \in V(G): c(v)=x\}$ is a stable set in $G$ for every $x \in [s+3,2s+1]$, a contradiction.

So $Y_1 \neq \emptyset$.
Suppose that $N_G(Y_1)=\emptyset$.
Let $G'=G-Y_1$.
Then $L|_{G'}$ is an $(s,\emptyset,0,0)$-list-assignment of $G'$ when $G$ has no $H$-minor; $L|_{G'}$ is an $(s,\emptyset,\ell,s-1)$-list-assignment of $G'$ when $G$ has no odd $H$-minor.
By the minimality of $G$, there exists an $(\eta,g)$-bounded $L|_{G'}$-coloring $c$ of $G'$ such that if $G$ has no odd $H$-minor, then $\{v \in V(G'): c(v)=x\}$ is a stable set in $G'$ for every $x \in [s+3,2s+1]$.
By further coloring each vertex $y$ in $Y_1$ with the unique element in $L(y)$, since $\lvert Y_1 \rvert \leq \lvert Y_1 \rvert^2g(\lvert Y_1 \rvert)$, there exists an $(\eta,g)$-bounded $L$-coloring of $G$ such that if $G$ has no odd $H$-minor, then $\{v \in V(G'): c(v)=x\}$ is a stable set in $G'$ for every $x \in [s+3,2s+1]$, a  contradiction.
This proves the claim.
\end{proof}

\begin{claim}\label{OddMinorFreeClaim2}
 $\lvert Y_1 \rvert \geq 9\theta'+1$.
 \end{claim}

\begin{proof}
Suppose $\lvert Y_1 \rvert \leq 9\theta'$.
So $\lvert Y_1 \rvert <\eta_1$.

Since $\eta \geq \eta_1$ and $g \geq g_1$, if there exists an $(\eta_1,g_1)$-bounded $L$-coloring of $G$ such that for every $x \in [s+3,2s+1]$, the set of vertices colored $x$ is a stable set in $G$, then it is an $(\eta,g)$-bounded $L$-coloring of $G$ such that for every $x \in [s+3,2s+1]$, the set of vertices colored $x$ is a stable set in $G$, a contradiction.
So there exists no $(\eta_1,g_1)$-bounded $L$-coloring of $G$ such that for every $x \in [s+3,2s+1]$, the set of vertices colored $x$ is a stable set in $G$.

Hence, by \cref{enlarge Y} (with taking $F=[s+3,2s+1]$) and \cref{OddMinorFreeClaim1}, there exist $Y_1' \subseteq V(G)$ with $\lvert Y_1 \rvert < \lvert Y_1' \rvert \leq \eta_1$ and an $(s,r+2,Y_1')$-list-assignment $L'$ of $G$ with $L'(v) \subseteq L(v)$ for every $v \in V(G)$ such that:
	\begin{itemize}
		\item there exists no $(\eta_1,g_1)$-bounded $L'$-coloring of $G$ such that for every $x \in [s+3,2s+1]$, the set of vertices colored $x$ is a stable set in $G$, 
		\item for every $L'$-coloring of $G$, every monochromatic component intersecting $Y_1$ is contained in $G[Y_1']$,
		\item $\{y \in Y_1: L(y) \cap [s+3,2s+1] \neq \emptyset\} = \{y \in Y_1': L'(y) \cap [s+3,2s+1] \neq \emptyset\}$,
		\item for every $x \in \{\ell\} \cup [s+3,2s+1]$ and $y \in Y_1'$ with $x \in L'(y)$, we have $\{v \in N_G(y)-Y_1': x \in L'(v)\}=\emptyset$, and 
		\item for every $v \in V(G)-Y_1'$, we have $L'(v) \cap [s+3,2s+1] = L(v) \cap [s+3,2s+1]$.
	\end{itemize}
So $L'$ is an $(s,Y_1',0,0)$-list-assignment of $G$ when $G$ has no $H$-minor; $L'$ is an $(s,Y_1',\ell,s-1)$-list-assignment of $G$ when $G$ has no odd $H$-minor.

Since $\eta \geq \lvert Y_1' \rvert > \lvert Y_1 \rvert$, the maximality of $Y_1$ implies that there exists an $(\eta,g)$-bounded $L'$-coloring $c'$ of $G$ such that for every $x \in [s+3,2s+1]$, the set of vertices colored $x$ is a stable set in $G$.
So every monochromatic component with respect to $c'$ contains at most $\eta^2g(\eta)$ vertices.
Since $c'$ is an $L'$-coloring, every monochromatic component with respect to $c'$ intersecting $Y_1$ is contained in $G[Y_1']$ and hence contains at most $\lvert Y_1' \rvert \leq \eta_1 \leq \lvert Y_1 \rvert^2g(\lvert Y_1 \rvert)$ vertices.
Since $L'(v) \subseteq L(v)$ for every $v \in V(G)$, $c'$ is an $L$-coloring.
Therefore, $c'$ is an $(\eta,g)$-bounded $L$-coloring of $G$ such that for every $x \in [s+3,2s+1]$, the set of vertices colored $x$ is a stable set in $G$, a contradiction.
\end{proof}

Define $\T=\{(A,B): \lvert V(A \cap B) \rvert <\theta', \lvert V(A) \cap Y_1 \rvert \leq 3\theta'\}$ to be a set of separations of $G$.

\begin{claim} \label{OddMinorFreeClaim3}
$\T$ is a tangle in $G$ of order $\theta'$.
\end{claim}

\begin{proof}
Suppose that $\T$ is not a tangle in $G$ of order $\theta'$.
Since $G$ has no $K_{s,t}$ subgraph and $L$ is an $(s,r+2,Y_1)$-list-assignment of $G$ with $\eta \geq \lvert Y_1 \rvert \geq 9\theta'+1$ by \cref{OddMinorFreeClaim2}.
By \cref{coloring or tangle} by taking $s=s,t=t,\theta=\theta',\eta=\eta,g=g,r=r+2$ and $F=\{\ell'\} \cup [s+3,2s+1]$, there exists an induced subgraph $G'$ of $G$ with $\lvert V(G') \rvert < \lvert V(G) \rvert$, a subset $Y_1' \subseteq V(G')$ with $\lvert Y_1' \rvert \leq \eta$, and an $(s,r+2,Y_1')$-list-assignment $L'$ of $G'$ with $L'(v) \subseteq L(v)$ for every $v \in V(G)$ such that:
	\begin{itemize}
		\item there exists no $(\eta,g)$-bounded $L'$-coloring of $G'$ such that for every $x \in [s+3,2s+1]$, the set of vertices colored $x$ is a stable set,
		\item $\{v \in Y_1': (\{\ell'\} \cup [s+3,2s+1]) \cap L'(v) \neq \emptyset\} = \{v \in Y_1: (\{\ell'\} \cup [s+3,2s+1]) \cap L(v) \neq \emptyset\} \cap V(G')$, and 
		\item for every $v \in V(G')-Y_1'$, we have $L'(v) \cap (\{\ell'\} \cup [s+3,2s+1]) = L(v) \cap (\{\ell'\} \cup [s+3,2s+1])$.
	\end{itemize}
Hence $L'$ is an $(s,Y_1,\ell',r)$-list-assignment of $G'$.
This contradicts the minimality of $G$.
\end{proof}

\begin{claim}
\label{OddMinorFreeClaim4}
 $\T$ controls a $K_{t'}$-minor.
 \end{claim}

\begin{proof}
Suppose to the contrary that $\T$ does not control a $K_{t'}$ minor.
Note that $\theta' \geq \theta_0$, $\eta \geq \eta_0(\theta')$ and $g \geq g_0$.
By \cref{NotControllingMinor}, since there does not exist an $(\eta,g)$-bounded $L$-coloring of $G$ such that for every $x \in [s+3,2s+1]$, the set of vertices color with $x$ is a stable set in $G$, we know there exist $(A^*,B^*) \in \T$, a set $Y_{A^*}$ with $\lvert Y_{A^*} \rvert \leq \eta_0(\theta') \leq \eta$ and $Y_1 \cap V(A^*) \subseteq Y_{A^*} \subseteq V(A^*)$, and an $(s,Y_{A^*},\ell',r)$-list-assignment $L_{A^*}$ of $G[V(A^*)]$ such that there exists no $(\eta,g)$-bounded $L_{A^*}$-coloring of $G[V(A^*)]$ such that for every $x \in [s+3,2s+1]$, the set of vertices colored $x$ is a stable set in $G[V(A^*)]$.
But $\lvert V(A^*) \rvert < \lvert V(G) \rvert$ since $\lvert V(A^*) \cap Y_1 \rvert \leq 3\theta' < \lvert Y_1 \rvert$.
So it contradicts the minimality of $G$.
\end{proof}

By \cref{OddMinorFreeClaim4}, $G$ contains a $K_{t'}$-minor $\alpha$.
In particular, $G$ has no odd $H$-minor and $L$ is an $(s,Y_1,\ell,s-1)$-list-assignment of $G$.

By \cref{structure odd minor}, either $G$ contains an odd $K_{\lvert V(H) \rvert}$-minor, or there exists a set $Z \subseteq V(G)$ with $\lvert Z \rvert < 8\lvert V(H) \rvert \leq \xi$ such that the unique block $U$ of $G-Z$ intersecting all branch vertices of $\alpha$ disjoint from $Z$ is bipartite.
The former implies that $G$ contains an odd $H$-minor, a contradiction.
So the latter holds.

Let $(A_0,B_0)$ be the separation of $G$ with $V(A_0 \cap B_0)=Z$ and $U \subseteq B_0$, and subject to this, $\lvert V(A_0) \rvert$ is maximal, and subject to these, $\lvert E(A_0) \rvert$ is minimal.
Note that $A_0-Z$ is the union of all components of $G-Z$ disjoint from $U$.
For each cut-vertex $v$ of $G-Z$ contained in $U$, let $(A_v,B_v)$ be the separation of $G$ with $V(A_v \cap B_v)=Z \cup \{v\}$ and $A_0 \cup U \subseteq B_v$, and subject to this, $\lvert V(A_v) \rvert$ is maximal, and subject to these, $\lvert E(A_v) \rvert$ is minimal.
Define $\L=\{(A_0,B_0),(A_v,B_v): v$ is a cut-vertex of $G-Z$ contained in $U\}$.
Then $\L$ is a location by the maximality of $V(A_0)$ and $V(A_v)$ and the minimality of $\lvert E(A_0) \rvert$ and $\lvert E(A_v) \rvert$.
Note that for every $(A,B) \in \L$, the order of $(A,B)$ is at most $\lvert Z \rvert+1 \leq \xi < t'<\theta'$, so either $(A,B) \in \T$ or $(B,A) \in \T$ by \cref{OddMinorFreeClaim3}.

\begin{claim}
\label{OddMinorFreeClaim5}
$\L \subseteq \T$.
\end{claim}

\begin{proof}
Suppose that $\L \not \subseteq \T$.
Let $(A,B) \in \L-\T$.
Then either $(A,B)=(A_0,B_0)$ and $(B_0,A_0) \in \T$, or there exists a cut-vertex $v$ of $G-Z$ contained in $U$ such that $(A,B)=(A_v,B_v)$ and $(B_v,A_v) \in \T$.
Since $t' \geq \lvert Z \rvert+2$, there exist vertices $u_1,u_2$ of $H$ such that $Q_1,Q_2$ are disjoint from $Z$, where $Q_1,Q_2$ are the branch sets of $\alpha$ corresponding to $u_1,u_2$, respectively.
By the definition of $U$, $U$ intersects both $Q_1$ and $Q_2$.
Since $Q_1 \cap Q_2=\emptyset$ and $U \subseteq V(B)$, one of $Q_1,Q_2$ is contained in $B-V(A)$. 
Since $\T$ controls $\alpha$ and $(B,A) \in \T$ has order less than $t'$, both $Q_1$ and $Q_2$ intersect $V(A)-V(B)$, a contradiction.
Hence $\L \subseteq \T$.
\end{proof}

Let $L'$ be a $(Z,\ell)$-growth of $L$, and let $Y_1' = \{v \in V(G): \lvert L'(v) \rvert=1\}$.
By \cref{Z growth}, $\lvert Y_1' \rvert \leq h(\lvert Y_1 \cup Z \rvert) \leq h(\eta+\xi)$ and for every $(A,B) \in \L$, we have $\lvert Y_1' \cap V(A) \rvert \leq h(\lvert A \cap B \rvert + \lvert Y_1 \cap V(A) \rvert) \leq h(4\theta')$.

\begin{claim}
\label{OddMinorFreeClaim6}
For every $v \in V(G)-Y_1'$, we have $L'(v) \cap (\{\ell\} \cup [s+3,2s+1]) \neq \emptyset$ and $L'(v)-(\{\ell\} \cup [s+3,2s+1]) \neq \emptyset$.
\end{claim}

\begin{proof}
By \cref{Z growth}, $L'$ is an $(s,Y_1',\ell,s-1)$-list-assignment of $G$.
So $L'$ is an $(s,s+1,Y_1')$-list-assignment of $G$.
By (L5), $\lvert L'(v) \rvert \geq s+2$ for every $v \in V(G)-Y_1'$.
Since $\lvert \{\ell\} \cup [s+3,2s+1] \rvert = s$, we have $L'(v) - (\{\ell\} \cup [s+3,2s+1]) \neq \emptyset$ for every $v \in V(G)-Y_1'$.
Since $\lvert [2s+1]- (\{\ell\} \cup [s+3,2s+1]) \rvert = s+1$, we have $L'(v) \cap (\{\ell\} \cup [s+3,2s+1]) \neq \emptyset$ for every $v \in V(G)-Y_1'$.
\end{proof}

Let $G' = G[(\bigcap_{(A,B) \in \L}V(B)) \cup Y_1']$.
Note that $G'=G[V(U) \cup Y_1']$.
Observe that $G'-Y_1' \subseteq U$, so $G'-Y_1'$ is bipartite.
Let $\{P,Q\}$ be a bipartition of $G'-Y_1'$.
Define an $L'|_{G'}$-coloring $c'$ of $G'$ as follows.
\begin{itemize}
	\item If $v \in Y_1'$, then let $c'(v)$  be the unique element of $L'(v)$.
	\item If $v \in P$, then let $c'(v)$  be an element of $L'(v) \cap (\{\ell\} \cup [s+3,2s+1])$.
	\item If $v \in Q$, then let $c'(v)$ be an element of $L'(v)-(\{\ell\} \cup [s+3,2s+1])$.
\end{itemize}
Note that $c'$ is well-defined by \cref{OddMinorFreeClaim6}.

\begin{claim}
\label{OddMinorFreeClaim7}
Every monochromatic component with respect to $c'$ contains at most $\eta_2$ vertices.
Furthermore, for every $x \in [s+3,2s+1]$, the set of vertices colored $x$ is a stable set in $G'$.
\end{claim}

\begin{proof}
Let $x \in [2s+1]$, and let $M$ be a monochromatic component with respect to $c'$ such that all vertices of $M$ are colored $x$.
If $V(M) \cap Y_1' = \emptyset$, then $V(M) \subseteq P$ or $V(M) \subseteq Q$, so $M$ consists of one vertex as $P$ and $Q$ are stable sets in $G'$.
So we may assume that $V(M) \cap Y_1' \neq \emptyset$.

Since $L'$ is an $(s,Y_1',\ell,s-1)$-list-assignment of $G$, by (R3), if $x \in \{\ell\} \cup [s+3,2s+1]$, then either $V(M) \cap Y_1'=\emptyset$ or $V(M) \subseteq Y_1'$.
Recall $V(M) \cap Y_1' \neq \emptyset$, so if $x \in \{\ell\} \cup [s+3,2s+1]$, then $V(M) \subseteq Y_1'$.
Hence, if $x \in [s+3,2s+1]$, then $V(M) \subseteq Y_1'$, so $M$ consists of one vertex by (R4); if $x=\ell$, then $V(M) \subseteq Y_1'$, so $M$ contains at most $\lvert Y_1' \rvert \leq h(\eta+\xi) \leq \eta_2$ vertices.

So we may assume that $x \in [s+2]-\{\ell\}$.
In particular, $V(M) \cap P = \emptyset$.
Since $L'$ is an $(s,Y_1',\ell,s-1)$-list-assignment, $L'$ is an $(s,s+1,Y_1')$-list-assignment, so $\lvert N_G(v) \cap Y_1' \rvert \geq s$ for every vertex $v \in V(M)-Y_1'$ with $N_G(v) \cap Y_1' \cap V(M) \neq \emptyset$ by (L3).
By \cref{BoundedGrowth}, $\lvert \{v \in V(M)-Y_1': N_G(v) \cap Y_1' \cap V(M) \neq \emptyset\} \rvert \leq f(\lvert Y_1' \rvert)$.
Note that every component of $M-Y_1'$ is contained in $G'[Q]$ which is a graph with no edge.
Since $V(M) \cap Y_1' \neq \emptyset$, every vertex in $V(M)-Y_1'$ has a neighbor in $Y_1' \cap V(M)$.
So $\{v \in V(M)-Y_1': N_G(v) \cap Y_1' \cap V(M) \neq \emptyset\} = V(M)-Y_1'$.
Hence $\lvert V(M) \rvert \leq \lvert V(M) \cap Y_1' \rvert + \lvert V(M)-Y_1' \rvert \leq \lvert Y_1' \rvert + f(\lvert Y_1' \rvert) \leq h(\eta+\xi) + f(h(\eta+\xi)) =\eta_2$.
\end{proof}

Let $Y_1'': = V(G')$.
Note that $Y_1'' \supseteq Y_1'$.
Let $L''$ be the following list-assignment of $G$:
	\begin{itemize}
		\item For every $v \in Y_1'$, let $L''(v):=L'(v)$.
		\item For every $v \in V(G')-Y_1'$, let $L''(v):=\{c'(v)\}$.
		\item For every $v \in N_G^{<s}(Y_1'')$, let $L''(v)$ be a subset of $L'(v)-\{L''(y): y \in Y_1'' \cap N_G(v)\}$ of size $\lvert L'(v) \rvert - \lvert N_G(v) \cap (Y_1''-Y_1') \rvert$ such that $\lvert (L'(v)-L''(v)) \cap [s+3,2s+1] \rvert = \lvert \{y \in N_G(v) \cap Y_1''-Y_1': L''(y) \in [s+3,2s+1]\} \rvert$ if possible.
		\item For every $v \in V(G)-(Y_1'' \cup N_G^{<s}(Y_1''))$, let $L''(v):=L'(v)$. 
	\end{itemize}

\begin{claim}
\label{OddMinorFreeClaim8}
The following hold:
	\begin{itemize}
		\item $L''$ is an $(s,s+1,Y_1'')$-list-assignment.
		\item For every $x \in [s+3,2s+1]$, $\{y \in Y_1'': x \in L''(y)\} = \{y \in Y_1': x \in L'(y)\} \cup \{y \in V(G')-Y_1': c'(y)=x\}$ is a stable set in $G$.
		\item For every $v \in N_G^{<s}(Y_1'') \cup (V(G)-N_G[Y_1''])$, we have $(s-1) - \lvert L''(v) \cap [s+3,2s+1] \rvert = \lvert \{y \in N_G(v) \cap Y_1'': L''(y) \subseteq [s+3,2s+1]\} \rvert$. 
		\item For every $(A,B) \in \L$, we have $\lvert Y_1'' \cap V(A) \rvert \leq h(4\theta')+1$.
	\end{itemize}
\end{claim}

\begin{proof}
Since $L'$ is an $(s,s+1,Y_1')$-list-assignment, $L''$ is an $(s,s+1,Y_1'')$-list-assignment.
And by \cref{OddMinorFreeClaim7}, for every $x \in [s+3,2s+1]$, 
\begin{align*}
\{y \in Y_1'': x \in L''(y)\} & = \{y \in Y_1': x \in L'(y)\} \cup \{y \in V(G')-Y_1': c'(y)=x\} \\
& = \{y \in V(G'): c'(y)=x\}
\end{align*} is a stable set in $G$.

Note that for every $v \in V(G)-Y_1''$, there exists $u \in U \cup \{0\}$ such that $v \in V(A_u)-V(B_u)$, so for every $y \in N_G(v) \cap V(G')-Y_1'$, $u \in U$ and $y$ must be $u$, and hence $\lvert N_G(v) \cap Y_1''-Y_1' \rvert \leq 1$.

For every $v \in N_G^{<s}(Y_1'')$, if $N_G(v) \cap Y_1'' = N_G(v) \cap Y_1'$, then $L''(v) \cap [s+3,2s+1]  = L'(v) \cap [s+3,2s+1]$; if $N_G(v) \cap Y_1'' \neq N_G(v) \cap Y_1'$, then $\lvert N_G(v) \cap Y_1' \rvert \leq s-2$, and by (L3) and (L4), we know $\lvert L'(v) \rvert = 2s+1-\lvert N_G(v) \cap Y_1' \rvert \geq s+3$, so $L'(v) \cap [s+2] \neq \emptyset \neq L'(v) \cap [s+3,2s+1]$, and since $\lvert N_G(v) \cap Y_1''-Y_1' \rvert \leq 1$, we know that $L''(v)$ can be chosen such that $\lvert (L'(v)-L''(v)) \cap [s+3,2s+1] \rvert = \lvert \{y \in N_G(v) \cap Y_1''-Y_1': L''(y) \in [s+3,2s+1]\} \rvert$.

Therefore, for every $v \in N_G^{<s}(Y_1'')$, 
\begin{align*}
	& (s-1) - \lvert L''(v) \cap [s+3,2s+1] \rvert\\ 
	= \;& (s-1) -\lvert L'(v) \cap [s+3,2s+1] \rvert  + \lvert (L'(v)-L''(v)) \cap [s+3,2s+1] \rvert\\ 
	= \;&  \lvert \{y \in N_G(v) \cap Y_1': L'(y) \subseteq [s+3,2s+1]\}\rvert + \lvert (L'(v)-L''(v)) \cap [s+3,2s+1] \rvert \\
	= \;&  \lvert \{y \in N_G(v) \cap Y_1': L'(y) \subseteq [s+3,2s+1]\}\rvert \\
	& \quad + \lvert \{y \in N_G(v) \cap Y_1''-Y_1': L''(y) \in [s+3,2s+1]\} \rvert \\
	= \;&  \lvert \{y \in N_G(v) \cap Y_1'': L''(y) \subseteq [s+3,2s+1]\} \rvert,
\end{align*} 
where the second equality follows from the fact that $L'$ satisfies (R5). 

For every $v \in V(G)-N_G[Y_1'']$, since $L''(v)=L'(v)$, by (R5),  
\begin{align*}
	(s-1)-\lvert L''(v) \cap [s+3,2s+1] \rvert  
	& = (s-1)- \lvert L'(v) \cap [s+3,2s+1] \rvert \\ 
	& = \lvert \{y \in N_G(v) \cap Y_1': L'(y) \subseteq [s+3,2s+1]\} \rvert \\
	& =0 = 
	\lvert \{y \in N_G(v) \cap Y_1'': L''(y) \subseteq [s+3,2s+1]\} \rvert.
\end{align*}
For every $(A,B) \in \L$, we have $\lvert Y_1'' \cap V(A) \rvert \leq \lvert Y_1' \cap V(A) \rvert + \lvert (Y_1''-Y_1') \cap V(A) \rvert \leq h(4\theta') + \lvert V(A \cap B)-Z \rvert \leq h(4\theta') + 1$.
\end{proof}

Let $L'''$ be a $(N_G^{\geq s}(Y_1''), \{\ell\} \cup [s+3,2s+1])$-progress of $L''$. 
Let $Y_1''':=\{v \in V(G): \lvert L'''(v) \rvert=1\}$.
By \cref{progress basic} and \cref{OddMinorFreeClaim8}, $L'''$ is an $(s,s+1,Y_1''')$-list-assignment of $G$ and satisfies (R3)-(R5).
Hence $L'''$ is an $(s,Y_1''',\ell,s-1)$-list-assignment.
Furthermore, for every $(A,B) \in \L$, since $V(A \cap B) \subseteq Y_1''$, we have $\lvert Y_1''' \cap V(A) \rvert \leq \lvert Y_1'' \cap V(A) \rvert + f(\lvert Y_1'' \cap V(A) \rvert)$ by \cref{BoundedGrowth}, so $\lvert Y_1''' \cap V(A) \rvert \leq h(4\theta')+1 + f(h(4\theta')+1) \leq \eta$ by 
\cref{OddMinorFreeClaim8}.

Let $L^*$ be an $(\emptyset,\ell)$-growth of $L'''$, and let $Y_1^*:=\{v \in V(G): \lvert L^*(v) \rvert=1\}$.
By \cref{Z growth}, $L^*$ is an $(s,Y_1^*,\ell,s-1)$-list-assignment of $G$ such that:		
	\begin{itemize}
		\item For every $L^*$-coloring of $G$, every monochromatic component intersecting $Y_1'''$ is contained in $G[Y_1^*]$. 
		\item For every $(A,B) \in \L$, we have $\lvert Y_1^* \cap V(A) \rvert \leq h(\lvert V(A \cap B) \rvert + \lvert Y_1''' \cap V(A) \rvert) \leq h(\theta'+h(4\theta')+1 + f(h(4\theta')+1)) \leq \eta_3 \leq \eta$.
	\end{itemize}

For every $(A,B) \in \L$, since $V(A \cap B) \subseteq Y_1^*$, we have $L^*|_{G[V(A)]}$ is an $(s,Y_1^* \cap V(A),\ell,s-1)$-list-assignment of $G[V(A)]$ such that $\lvert \{y \in V(A): \lvert L^*(y) \rvert=1\} \rvert = \lvert Y_1^* \cap V(A) \rvert \leq \eta_3 \leq \eta$.
For every $(A,B) \in \L$, since $(A,B) \in \T$, we have $\lvert V(A) \rvert < \lvert V(G) \rvert$.
Therefore, for every $(A,B) \in \L$, by the minimality of $G$, there exists an $(\eta,g)$-bounded $L^*|_{G[V(A)]}$-coloring $c_A^*$ of $G[V(A)]$ such that $\{v \in V(A): c_A^*(v)=x\}$ is a stable set in $G[V(A)]$ for every $x \in [s+3,2s+1]$. 

Since $\bigcap_{(A,B) \in \L}V(B) \subseteq V(G') \subseteq Y_1^*$ and $\L$ is a location, for every $v \in V(G)-Y_1^*$, there uniquely exists $(A_v,B_v) \in \L \subseteq \T$ such that $v \in V(A_v)-V(B_v)$. 
Let $c^*$ be the following function:
	\begin{itemize}
		\item For every $v \in Y_1^*$, let $c^*(v)$ be the unique element in $L^*(v)$.
		\item For every $v \in V(G)-Y_1^*$, let $c^*(v):=c^*_{A_v}(v)$.
	\end{itemize}
Clearly, $c^*$ is well-defined and is an $L^*$-coloring of $G$.

Suppose that there exists $x^* \in [s+3,2s+1]$ such that the set $\{v \in V(G): c^*(v)=x^*\}$ is not a stable set in $G$.
Then there exists an edge $e$ of $G$ such that both ends of $e$ are colored $x^*$ under $c^*$.
If there exists $(A,B) \in \L$ such that $e$ belongs to $A$, then $\{v \in V(A): c_A^*(v)=x^*\}$ is not a stable set in $G[V(A)]$, a contradiction.
So $e \in \bigcap_{(A,B) \in \L}B \subseteq G'$, and hence $\{v \in V(G'): c'(v)=x^*\}$ is not a stable set in $G'$, contradicting \cref{OddMinorFreeClaim7}.

This shows that for every $x \in [s+3,2s+1]$, the set $\{v \in V(G): c^*(v)=x\}$ is a stable set in $G$.
Hence $c^*$ is not $(\eta,g)$-bounded.

Since $L^*(v) \subseteq L'(v)$ for every $v \in V(G)$, $c^*$ is an $L'$-coloring of $G$.
Since $L'$ is an $(Z,\ell)$-growth of $L$, by \cref{Z growth}, every monochromatic component with respect to $c^*$ intersecting $Y_1 \cup Z$ is contained in $G[Y_1']$.
So the union of monochromatic components with respect to $c^*$ intersecting $Y_1 \cup Z$ is contained in $G[Y_1']$ and hence contains at most $\lvert Y_1' \rvert \leq  h(\eta + \xi) \leq \eta_2 \leq g(0) \leq g(\lvert Y_1 \rvert) \leq g(\eta)$. 
In particular, the union of monochromatic components with respect to $c^*$ intersecting $Y_1$ contains at most $\lvert Y_1 \rvert^2g(\lvert Y_1 \rvert)$ vertices.

Hence there exists a monochromatic component $M$ with respect to $c^*$ disjoint from $Y_1 \cup Z$ such that $M$ contains more than $\eta^2g(\eta)$ vertices.
Suppose that $V(M) \cap V(G') \neq \emptyset$.
Since $V(M) \cap Z=\emptyset$ and $\lvert V(A \cap B)-Z \rvert \leq 1$ for every $(A,B) \in \L$, $G[V(M) \cap V(G')]$ is connected.
Since $c^*(v)=c'(v)$ for every $v \in V(G')$, $M \cap G'$ is a monochromatic component with respect to $c'$, so $\lvert V(M \cap G') \rvert \leq \eta_2$ by \cref{OddMinorFreeClaim7}. 
For every $(A,B) \in \L$, if $V(M)-V(B) \neq \emptyset$, then $V(M) \cap V(A \cap B)-Z \neq \emptyset$ since $V(M)$ is disjoint from $Z$.
By the definition of $\L$, if $(A_1,B_1),(A_2,B_2)$ are distinct members of $\L$, then $V(A_1 \cap B_1)-Z \neq V(A_2 \cap B_2)-Z$.
Let $\L'=\{(A,B) \in \L: V(M) \cap V(A \cap B)-Z \neq \emptyset\}$.
So $\lvert \L' \rvert \leq \lvert V(M) \cap V(G') \rvert \leq \eta_2$.
For every $(A,B) \in \L'$, since $V(M)$ intersects $V(A \cap B)-Z \subseteq Y_1''' \cap V(A)$, $G[V(M) \cap V(A)]$ is a monochromatic component with respect to $c_A^*$ intersecting $Y_1''' \cap V(A) \subseteq Y_1^* \cap V(A)$, so by \cref{Z growth}, $G[V(M) \cap V(A)]$ is contained in $G[Y_1^* \cap V(A)]$ and hence contains at most $\lvert Y_1^* \cap V(A) \rvert \leq \eta_3$ vertices. 
Therefore, $\lvert V(M) \rvert \leq \lvert V(M) \cap V(G') \rvert \cdot \eta_3 \leq \eta_2\eta_3 \leq g(0) \leq \eta^2g(\eta)$, a contradiction.

Hence $V(M) \cap V(G')=\emptyset$.
In particular, $V(M)-V(B) \neq \emptyset$ for some $(A,B) \in \L$.
For every $(A',B') \in \L$, since $V(A' \cap B') \subseteq \bigcap_{(A,B)\in \L}V(B) \subseteq V(G')$, $V(M)$ is disjoint from $V(A' \cap B')$.
So if there exist $(A_1,B_1),(A_2,B_2) \in \L$ such that $V(M)-V(B_1) \neq \emptyset \neq V(M)-V(B_2)$, then since $\L$ is a location and $M$ is connected, $V(M)$ intersects $V(A_1 \cap B_1)$, a contradiction. 
So there uniquely exists $(A^*,B^*) \in \L$ such that $V(M)-V(B^*) \neq \emptyset$.
Since $V(M)$ is disjoint from $V(A^* \cap B^*)$, we have $V(M) \subseteq V(A^*)-V(B^*)$.
Hence $M$ is a monochromatic component with respect to $c^*_{A^*}$.
Since $c^*_{A^*}$ is an $(\eta,g)$-bounded $L^*|_{G[V(A^*)]}$-coloring, $M$ contains at most $\eta^2g(\eta)$ vertices, a contradiction.
This proves the lemma.
\end{proof}

\section{About the Gerards-Seymour Conjecture}
\label{GerardsSeymour}

This section proves \cref{OddMinor}. The proof depends on the following result of \citet{KO19}. 
For graphs $G$ and $H$, let $G+H$ be the graph obtained from the disjoint union of $G$ and $H$ by adding all edges with one end in $V(G)$ and one end in $V(H)$.
Let $K^*_{s,t}:= K_s + I_t$, where $I_t$ is the graph on $t$ vertices with no edges. 

\begin{lemma}[Corollary of {\cite[Lemma 5.1]{KO19}}] 
\label{k-o odd minor}
Let $t,d\in\mathbb{N}$ with $d\geq 3$. 
If there exists a positive integer $\eta \geq 4t-3$ such that every odd $K_{t+1}$-minor-free graph containing no bipartite 
$K^*_{2t,t+1}$-subdivision has a $d$-coloring with clustering $\eta$, then every odd $K_{t+1}$-minor-free graph has a $(d+4t-3)$-coloring with clustering $\eta$. 
\end{lemma}

\begin{proof}
Let $\F$ be the set of all graphs whose every component contains at most $\eta$ vertices.
So $\F$ is a class of graphs closed under isomorphism and taking disjoint union, and $\F$ contains all graphs on at most $4t-3$ vertices.
Since every graph with no odd $K_{t+1}$-minor and no bipartite 
$K^*_{2t,t+1}$-subdivision has a $d$-coloring with clustering $\eta$, by {\cite[Lemma 5.1]{KO19}}, every odd $K_{t+1}$-minor-free graph has a $(d+4t-3)$-coloring with clustering $\eta$.
\end{proof}

We now prove \cref{OddMinor}.

\begin{theorem}
For every $s\in\mathbb{N}$, there exists $\eta\in\mathbb{N}$ such that every odd $K_{s+1}$-minor-free graph has an $(8s-4)$-coloring with clustering $\eta$.
\end{theorem}

\begin{proof}
By \cref{oddminorbasic}, there exists $\eta$ such that every odd $K_{s+1}$-minor-free graph with no $K_{2s-1,{2s-1 \choose 2}+s+2}$ subgraph has a $(2(2s-1)+1)$-coloring with clustering $\eta$.
So every odd $K_{s+1}$-minor-free graph with no bipartite $K^*_{2s,s+1}$-subdivision
has a $(4s-1)$-coloring with clustering $\eta$. 
We may assume that $\eta \geq 4s-3$.
The theorem now follows from \cref{k-o odd minor}.
\end{proof}

\subsubsection*{Acknowledgements} 
Many thanks to the referees for their careful reading and helpful suggestions, especially to the referee who found a subtle error in the original submission.

  \let\oldthebibliography=\thebibliography
  \let\endoldthebibliography=\endthebibliography
  \renewenvironment{thebibliography}[1]{%
    \begin{oldthebibliography}{#1}%
      \setlength{\parskip}{0.15ex}%
      \setlength{\itemsep}{0.15ex}%
  }{\end{oldthebibliography}}


\def\soft#1{\leavevmode\setbox0=\hbox{h}\dimen7=\ht0\advance \dimen7
	by-1ex\relax\if t#1\relax\rlap{\raise.6\dimen7
		\hbox{\kern.3ex\char'47}}#1\relax\else\if T#1\relax
	\rlap{\raise.5\dimen7\hbox{\kern1.3ex\char'47}}#1\relax \else\if
	d#1\relax\rlap{\raise.5\dimen7\hbox{\kern.9ex \char'47}}#1\relax\else\if
	D#1\relax\rlap{\raise.5\dimen7 \hbox{\kern1.4ex\char'47}}#1\relax\else\if
	l#1\relax \rlap{\raise.5\dimen7\hbox{\kern.4ex\char'47}}#1\relax \else\if
	L#1\relax\rlap{\raise.5\dimen7\hbox{\kern.7ex
			\char'47}}#1\relax\else\message{accent \string\soft \space #1 not
		defined!}#1\relax\fi\fi\fi\fi\fi\fi}

\end{document}